%% file: Petra_Staynova_-_Rose-Hulman_-_Final.tex
\documentclass{rhumj} 

\usepackage{amsmath}
\usepackage{amsfonts}
\usepackage{amssymb}
\usepackage{amsthm}
\usepackage{textcomp}
\usepackage{tipa}
\usepackage{fullpage}
\usepackage{graphicx}
\usepackage{color}
\usepackage{url}

\newcommand{\comments}[1]{}

\renewcommand{\leq}{\leqslant}
\renewcommand{\bar}{\overline}
\renewcommand{\geq}{\geqslant}

\newcommand{\w}{\omega}

\newcommand{\al}{\alpha}
\newcommand{\be}{\beta}
\newcommand{\g}{\gamma}
\newcommand{\G}{\Gamma}
\newcommand{\de}{\delta}
\newcommand{\e}{\epsilon}
\newcommand{\la}{\lambda}
\newcommand{\D}{\Delta}
\newcommand{\ka}{\kappa}
\newcommand{\si}{\sigma}

\newcommand{\R}{\mathbb{R}}

\newcommand{\cC}{\mathcal{C}}
\newcommand{\Q}{\mathbb{Q}}
\newcommand{\I}{\mathbb{I}}
\newcommand{\N}{\mathbb{N}}
\newcommand{\Sorg}{\mathbb{S}}

\newcommand{\cZ}{\mathcal{Z}}
\newcommand{\B}{\mathcal{B}}
\newcommand{\U}{\mathcal{U}}
\newcommand{\F}{\mathcal{F}}
\newcommand{\V}{\mathcal{V}}
\newcommand{\W}{\mathcal{W}}

\newcommand{\halo}{\mathring}

\newcommand{\Lin}{Lindel\"of }
\newcommand{\Lind}{Lindel\"of}
\newcommand{\Int}{\textrm{Int}}

\newcommand{\Xp}{X^{+}}
\newcommand{\Xm}{X^{-}}

\newcommand{\2}{\sqrt{2}}
\newcommand{\n}{\frac{1}{n}}

\newcommand{\ra}{\rightarrow}
\newcommand{\lar}{\leftarrow}

\swapnumbers
\newtheorem{thm} {Theorem}[section]
\newtheorem{propn}[thm]{Proposition}
\newtheorem{lem} [thm]{Lemma}
\newtheorem{cor}[thm]{Corollary}

\theoremstyle{definition}
\newtheorem{defn}[thm]{Definition}
\newtheorem{xmpl}[thm]{Example}
\newtheorem{xmpls}[thm]{Examples}

\newtheorem*{construction}{Construction}
\newtheorem*{cl}{Claim}

\newtheorem*{oq}{Open Question}

\theoremstyle{remark}
\newtheorem{rmk}[thm]{Remark}
\newtheorem{nt}[thm]{Note}
\newtheorem*{case}{Case}

\title{
A Comparison of Lindel\"of-type Covering Properties of Topological Spaces 
}

\author{
Petra Staynova \affiliation{Oxford University}
}

\authornames{
Petra Staynova
}

\acknowledgement{
The author is in debt to the referee for his careful and assiduous reading of the initially submitted paper. 
 The present version has greatly benefitted from his valuable comments and numerous suggestions for improvement. 
 The author is also thankful to the editor for his suggestions on improving the structure of this paper.\\*[\parskip]
\hspace*{1.5em}
}

\abstract{
Lindel\"of spaces are studied in any basic Topology course. 
However, there are other interesting covering properties with similar behaviour, such as almost \Lind, weakly \Lind, and quasi-\Lind,  that have been considered in various research papers. 
Here we present a comparison between the standard results on \Lin spaces and analogous results for weakly and almost \Lin spaces.
Some theorems, similar to the published ones, will be proved. 
We also consider counterexamples, most of which have not been included in the standard Topological textbooks,
that show the interrelations between 
those properties and various basic topological notions, such as separability, separation axioms, first countability, and others. 
Some new features of those examples will be noted in view of the present comparison. 
We also pose several open questions. 
}

\vol{12, No. 2}
\date{Fall 2011}

\begin{document}
\maketitle
\let\thefootnote\relax\footnotetext{This paper consists of the core part of author's 3d year BE Extended Essay at Oxford University, UK.}

\section{Historical Overview and Motivation}

One of the basic theorems in Real Analysis, the Heine--Borel Theorem, states (in modern terminology) that every closed interval on the real line is compact. 
Later it was discovered that a similar property holds in more general metric spaces: every closed and bounded subset turns out to be compact and conversely, every compact subset is closed and bounded. 
It turned out that compactness is in fact a covering property:
 the modern description of compactness via open covers emerged from the work of P. S. Alexandrov and P. S. Urysohn in their famous \emph{``Memoire sur les espaces topologiques compacts''} \cite{AU}. 
Compact spaces in many ways resemble finite sets. 
 For example, the fact that in Hausdorff topological spaces two different points can be separated by disjoint open sets easily generalizes to the fact that in such spaces two disjoint compact sets can also be separated by disjoint open sets. 
Any finite set is compact in any topology and the fact that compact spaces should be accompanied by some kind of separation axiom comes from the fact that any set with the co-finite topology is compact (but fails to be Hausdorff). 

 However, our favourite ``real'' objects such as the real line, real plane etc. fail to be compact. 
Ernst Lindel\"of was able to identify the first compactness-like covering property (which was later given his name): the property that from every open cover, one can choose a countable subcover.
The Lindel\"of theorem, stating that every second countable space is Lindel\"of, was proved by him for Euclidean spaces as early as 1903 in \cite{Lind}. 
Many facts that held for compact spaces, such as that every closed subspace of a compact space is also compact, remain true in Lindel\"of spaces. 
In metric spaces the Lindel\"of property was proved to be equivalent to separability, the existence of a countable basis and the countable chain condition -- all of which hold on the real line. 
But many other properties, such as preservation under products, fail even in the finite case for Lindel\"of spaces. 

Another generalization that is much closer to compactness is the notion of an H-closed space: this is a Hausdorff space in which from every open cover we can choose a finite subfamily with dense union. H-closed spaces turned out to be of broad interest and importance. And that is why it was quite natural when in 1959 Z.Frolik \cite{Fro} introduced and studied weakly Lindel\"of spaces that are generalisations of both H-closed and Lindel\"of spaces. 

Further attention to the \Lind-type covering properties was brought about by 
 the famous problem of P. S. Alexandrof whether any Hausdorff compact space with the first axiom of countability has cardinality at most continuum. 
This question remained unsolved for 50 years and only in 1969 did A. V. Arhangelskii  proved it to be true. 
But he proved much more: the cardinality of any Hausdorff first countable Lindelof space does not exceed continuum. 
The solution of that problem and some previous results of A. Hajnal and I. Juhasz gave birth to a whole new branch of topology: the theory of cardinal invariants of topological spaces. 
One central area of investigation there is to study what kind of topological restrictions of a space can ensure that its cardinality is at most continuum. 
Stated in that way, the question seems to be mainly of pure mathematical interest. 
But the fact that a given space has cardinality continuum is another interesting way both of defining various kinds of topologies on our old favourites: the real line, the real plane, etc., as well as of studying the variety of different topologies on $\R$ besides the well-known Eucledean one. 
 That is why further generalizations of the Lindel\"of property naturally appeared in the theory of cardinal invariants mainly as an attempt to generalize or weaken the conditions in Arhangelskii's result.
 Bell, Ginsbugh and Woods used weakly Lindel\"of spaces; Dissanayeke and Willard introduced almost Lindel\"of spaces; later on, Arhangelskii defined quasi-\Lin spaces. 
 It was not until recently that spaces satisfying Lindel\"of-type covering properties have been studied outside the context of cardinal functions and cardinality restrictions. 
 It  has been investigated whether these properties are preserved under the main operations on topological spaces: taking  subspaces, products, disjoint sums, quotient spaces, and so on.

\subsection{Project Structure}

In Section \ref{Lindelof} we will present some basic properties of Lindel\"of spaces.
These include main theorems and interrelations between separability, CCC, countable basis and the Lindel\"of property in metric spaces, as well as how separation axioms and Lindel\"ofness correlate in such spaces.
In Section \ref{Lindelofgeneralizations} we will also consider the main generalizations of the Lindel\"of property, such as weakly Lindel\"of spaces, almost Lindel\"of spaces and quasi-Lindel\"of spaces, and investigate how these properties relate to the corresponding properties of Lindel\"of spaces. 
Just as we loose finite productivity when going from compact to \Lin spaces, we also loose some other properties in the further generalisations of \Lind ness. 
That is why the main emphasis of this project will be on counterexamples that show failures of some properties.

We present several examples from journal papers, elaborating on the details, sometimes providing our own proofs, and proving, in addition, some new properties. 
 Some examples were constructed for a different purpose, but we use them in the new context of  interrelations between Lindel\"of-type covering properties. We will present some of them not in their most general form, but in a partial case that demonstrates their main ideas and allows for a more vivid presentation and visualisation. 

\subsection{Standard Notation, Definitions and Results}

We follow the basic terminology, results, and notions from \cite{SDb}, \cite{ENG}, and \cite{SW}. We assume all spaces to be Hausdorff unless otherwise stated. 
We will also consider the countable versions of notions and results proved in \cite{A}, \cite{BGW}, and \cite {DW}.
For counterexamples and less known theorems we use sources such as
 \cite{CS},
 \cite{SZ},
 and 
 \cite{HJ}. 
\\

For sets and topological spaces, we use:

\begin{itemize}
\item
$\tau_X$ for the topology on a set $X$;
\item
$(X,\tau_X)$ for the space $X$ with the respective topology $\tau$; or just $\tau$, when $X$ is clear; or just $X$, when $\tau$ is clear;
\item
$\halo{A}$ or $\Int(A)$ for the interior of $A\subset X$;
\item
$\bar{A}$  for the closure of $A\subset X$;
\item
$|A|$ for the cardinality of the set $A$;
\item
$\bigcup\U$ for $\displaystyle{\bigcup_{U\in\U}U}$, where $\U$ is some collection of sets;
\item
$d(A,x)$ - the distance between the point $x$ and the set $A$, defined as $d(A,x)=\inf\{d(a,x):a\in A\}$.
\end{itemize}

We write:
\begin{itemize}
\item
$\R$ for the reals;
\item
$\R^+$ and $\R^2_+$ - for the positive reals, and the open upper half-plane, respectively;
\item
$\N$ for the positive integers;
\item
$\Q$ for the rationals;
\item
$\I$ for the irrationals;
\item
$\mathcal{P}(X)$ for the power set of $X$;
\item
$\w_1$ for the first uncountable ordinal with its natural order. 
\end{itemize}

We always consider $\R$ with the Euclidean topology, unless otherwise stated.
\\

For convenience, finite sets will be considered countable. 
\\

From now on, let $X$ and $Y$ be a topological spaces.
\\

We use the following standard results and definitions from Topology:

\begin{defn} The following are standard topologies on $X$:
\begin{enumerate}
\item
The \emph{discrete topology} is $\tau_{discrete}=\mathcal{P}(X)$.
\item
The \emph{co-finite topology} is $\tau_{co-finite}=\{U \subset X: X \setminus U \textrm{ finite }\}$. 
\item
The \emph{co-countable topology} is $\tau_{co-countable}=\{U\subset X: X\setminus U\textrm{ countable }\}$.
\end{enumerate}
\end{defn}

We say that a set is \emph{clopen} if it is both closed and open. 

\begin{defn}
A collection $\B$ of open subsets of $X$ is called a \emph{basis}
 for the topology on $X$ if any open subset of $X$ is a union of some subfamily of $\B$. 
\end{defn}

\begin{propn} 
\label{basis}
Let $X$ be a set and $\B$ a family of subsets of $X$ such that 
\begin{itemize}
\item[(B1)]
 $X$ is a union of sets in $\B$;
\item[(B2)]
for any $B_1, B_2 \in \B$, 
there exists a $B_3\in\B$ such that $B_3\subset B_1\cap B_2$. 
\end{itemize}
Then $\B$ forms a basis for a topology on $X$.
\end{propn}

\begin{defn} Any open $U\subset X$ with $x\in U$ is called a \emph{neighborhood} of $x$. 
\end{defn}

\begin{defn}
The \emph{product topology} on $X\times Y$ is the topology generated by the basis $\displaystyle{\B=\{U\times V: U\in\tau_X, V\in\tau_Y\}}$.
\end{defn}

\begin{defn}
For any $x\in X$ the family of open sets $\V_x$ is called a \emph{local (neighborhood) basis} at the point $x$ if for any open  $U\ni x$ there is an element $V\in\V_x$ such that $x\in V\subset U$.
\end{defn}

\begin{propn} \label{LB} The family $\V_x$ of open subsets of $X$ is a local base at $x$ if and only if 
\begin{itemize}
\item[(LB1)] 
$x\in V$ for any $V\in\V_x$
\item[(LB2)]
for any $V_1,V_2\in\V_x$, there exists $V_3\in\V_x$ such that $V_3\subset V_1\cap V_2$. 
\end{itemize}
\end{propn}

\begin{propn}\label{LBB}
If $\B$ is a base for $X$, then $\B_x=\{V\in\V:x\in V\}$ is a local base for any $x\in X$. Conversely, $\B=\{\V_x:\V_x \textrm{ is a local base at x, }x\in X\}$ is a base for the topology on $X$. 
\end{propn}

\begin{defn} $X$ is \emph{first countable} if for any $x\in X$ there is a countable local base. 
\end{defn}

\begin{defn} $X$ is \emph{second countable} if its topology has a countable base. 
\end{defn}

\begin{defn} \label{CCC} $X$ satisfies the \emph{countable chain condition} if every family of non-empty disjoint open subsets of $X$ is countable. We say that $X$ is CCC.
\end{defn}

\begin{defn} We say that $X$ is \emph{separable} if it has a countable dense subset.
\end{defn}

\begin{defn} $X$ is \emph{$T_1$} if whenever $x\neq y\in X$, there is a neighborhood of each not containing the other.
\end{defn}

\begin{defn} $X$ is \emph{Hausdorff} if for every $x\neq y\in X$, there exist disjoint open sets $U,V\in X$ with $x\in U$ and $y\in V$. 
\end{defn}

\begin{defn} $X$ is \emph{Urysohn} if whenever $x\neq y\in X$ there are $U,\ V$ open with $x\in U$, $y \in V$, $\bar{U} \cap \bar{V} = \emptyset$. 
\end{defn}

\begin{defn} $X$ is \emph{regular} if whenever $F$ is closed and $x \notin F$, there exist disjoint open sets $U,\ V$ with $x \in U$ and $F \subset V$. 
\end{defn} 

\begin{propn} \label{regularlemma}
 $X$ is regular if whenever $U$ is open and $x \in U$,  there exists an open set $V$ with $x \in V \subset \bar{V} \subset U$.
\end{propn}

\begin{defn} $X$ is \emph{completely regular} if for any closed $F\subset X$ and $x\notin F$ there is a continuous function $f:X\rightarrow\R$ with $f(x)=1$ and $f(F)=0$. 
\end{defn}

\begin{propn}If $X$ is completely regular then $X$ is regular.
\end{propn}

\begin{defn}$X$ is \emph{normal} if whenever $F, \ K$ are disjoint closed sets, there exist disjoint open sets $U,\ V$ with $F \subset U$ and $K \subset V$.
\end{defn} 

\begin{propn} X is normal if whenever $F$ is closed, $U$ is open, and  $F \subset U$, there exists an open set $V$ with $F \subset V \subseteq \bar{V} \subset U$. 
\end{propn}

\begin{propn}
If $X$ is normal then it is completely regular. 
\end{propn}

\begin{defn}  $X$ is \emph{$T_4$} if it is a normal $T_1$-space.
\end{defn}

\begin{propn}
Every normal space is regular; every regular space is Urysohn; every Urysohn space is Hausdorff; every Hausdorff space is $T_1$.
\end{propn}

\begin{defn} $F\subset X$ is \emph{discrete} if whenever $U$ is open, then $|F\cap U|\leq 1$
\end{defn}

\begin{defn}
A map $f:X\rightarrow Y$ is called \emph{continuous} if the pre-image of any open set is open (or equivalently, the pre-image of every closed set is closed). 
\end{defn}

\begin{propn}\label{ctspropn} A map $f:X\rightarrow Y$ is  continuous if and only if $f(\bar{A})\subseteq \bar{f(A)}$ for every $A\subset X$. 
\end{propn}

\begin{defn} A map $f:X\rightarrow Y$ is \emph{open} if for every $U\in\tau_X$, $f(U)\in\tau_Y$.  
\end{defn}

\begin{propn} \label{opencts} 
 A continuous mapping $f: X \rightarrow Y$ is open if and only if there exists a base $\B$ for $X$ such that $f(U)$ is open in $Y$ for every $U \in \B$.
\end{propn}

\section{Lindel\"of Spaces} \label{Lindelof}

In this section, we will recollect some main definitions, will summarize the main results about \Lin spaces, and present in depth several examples about the interrelation of some main topological notions in \Lin spaces. 
We will also aim at providing detailed visualisation of those examples that will enable the reader to better grasp the core ideas.

\setcounter{subsection}{-1}
\subsection{Preliminaries}

\begin{defn} A topological space $X$ is called \emph{Lindel\"of} if every open cover of $X$ has a countable subcover. \end{defn}

\begin{xmpls}[] \label{Lsmpl} The following are straightforward examples of \Lin spaces:
\begin{enumerate}
\item
Any countable topological space;
\item
Any compact topological space;
\item
Any space with the co-countable topology;
\item
A countable union of compact spaces;
\item \label{RLind}
 $\R$,  with the Euclidean topology. 
\end{enumerate}
\end{xmpls}

\begin{xmpl}[]

Any uncountable discrete topological space is not Lindel\"of. 

\end{xmpl}

A more interesting example of a non-Lindel\"of space can be obtained by considering the order topology on $[0,\w_1)$ and $[0,\w_1]$. 
Let us recall that for any $\al\in(0,\w_1)$, open neighborhoods will be the open intervals $(\al_1,\al_2)\ni \al$, where $0<\al_1<\al_2<\w_1$;
open neighborhoods of $0$ will be $[0,\al)$ for every $\al<\w_1$; and open neighborhoods of $\w_1$ will be $(\al,\w_1]$ with $\al<\w_1$.
In this way we can topologize both $[0,\w_1)$ and $[0,\w_1]$. 

\begin{xmpl} \label{nLsubset} The space $[0,\w_1)$, with the order  topology,  is not Lindel\"of.
\end{xmpl} 

\begin{proof}
The family $\U=\{U_{\al}=[0,\al):\al\in[0,\w_1)\}$ is an open cover of $[0,\w_1)$ with no countable subcover. 
Indeed, suppose $\U'=\{U_{\al_1},\ldots,U_{\al_n},\ldots\}\subset\U$ covers $[0,\w_1)$. 
Then $\{\al_1,\ldots,\al_n,\ldots\}$ is a countable set of countable ordinals. 
Hence, $\al_0=\sup\{\al_1,\ldots,\al_n,\ldots\}<\w_1$. 
Hence, $[\al_0 +1,\w_1)$ remains uncovered by $\U'$ -- a contradiction. 
 \end{proof}

\begin{propn} \label{basic} $X$ is Lindel\"of if and only if every cover $\U$ with basic open sets has a countable subcover.
\end{propn}

\begin{defn}
We say that a family $\F$ of nonempty subsets of $X$ has the \emph{countable intersection property} if any countable subfamily of $\F$ has a non-empty intersection.
\end{defn}

There is the following useful characterization of the \Lin property:

\begin{thm} A space X is Lindel\"of if and only if every family of closed nonempty subsets of X which has the countable intersection property has a non-empty intersection. \end{thm} 

\begin{proof} 

Let $\F$ be a family of closed subsets of $X$ with the countable intersection property. 
Suppose that $\bigcap_{F\in\F}F=\emptyset$.
Then $\U=\{X\setminus F:F\in\F\}$ is an open cover for $X$.
Indeed, 
\begin{displaymath}
\bigcup_{U\in\U}U=\bigcup_{F\in\F}X\setminus F=X\setminus\bigcap_{F\in\F}F=X\setminus\emptyset=X.
\end{displaymath}
Since $X$ is Lindel\"of, there is a countable subcover $\U'\subset\U$; but this means that
\begin{displaymath}
\emptyset=X\setminus\bigcup_{U\in\U'}U=\bigcap_{U\in\U'}X\setminus U=\bigcap_{F\in\F'}F,
\end{displaymath}
contradicting the countable intersection property.
Hence $\bigcap_{F\in\F}F\neq\emptyset$.

Conversely, let $\U$ be an open cover of $X$. 
Then 
\begin{displaymath}
\F=\{X\setminus U: U\in\U\}
\end{displaymath}
 is a family of closed subsets of $X$. 
Suppose that $X$ is not Lindel\"of, so
for all countable subsets $\U'\subset\U$, there exists $x\in X$ with $x\in X\setminus\bigcup_{U\in\U'}U$.
Define 
\begin{displaymath}
\F'=\{X\setminus U: U \in \U'\}\subset\F.
\end{displaymath}
Then $x \in X\setminus \bigcup_{U\in\U'}U=\bigcap_{U\in\U'}(X\setminus U)=\bigcap_{F\in\F'}F$ - 
so $\F$ has the countable intersection property.
Thus we would have
\begin{displaymath}
\emptyset\neq\bigcap_{F\in\F}F=\bigcap_{U\in\U} X\setminus U=X\setminus\bigcap_{U\in\U}U=\emptyset
\end{displaymath}
- a contradiction. 
Hence, $X$ is Lindel\"of. 
\end{proof}

\subsection{The \Lin Property and the Main Topological Operations}

\begin{thm} 
\label{ClosedSubsetsofL}
Every closed subspace of a Lindel\"of space is Lindel\"of. \end{thm}

However, arbitrary subspaces of a Lindel\"of space need not be Lindel\"of. 

\begin{xmpl}
$[0,\w_1]$ is a Lindel\"of space, but $[0,\w_1)\subset [0,\w_1]$ is not Lindel\"of. 
 \end{xmpl}  

\begin{proof} 
 Let  $\U=\{ U_{\al}: \al \in A\}$ be an open cover for $[0,\w_1]$; then there exists $U_{\alpha_{0}}$ such that $\w_{1} \in U_{\al_{0}}$. 
Then $U_{\alpha_{0}}$ contains an interval $(\g, \w_{1}]$ for some $\g<\w_{1}$. 
This leaves possibly only the set $[0,\g ]$ uncovered, which is countable, so we need countably many more elements of $\U$. 

However, as we saw in Example \ref{nLsubset}
the subspace $[0,\w_1)=[0,\w_1] \backslash \{\w_{1} \}$ is not \Lind.
 \end{proof}

This example justifies the following definition: 

\begin{defn} If every subspace of a topological space $X$ is Lindel\"of, then X is called \textit{hereditarily Lindel\"of}. \end{defn}

\begin{propn}
\label{herLlem} If every open subspace of X is Lindel\"of, then X is hereditarily Lindel\"of.
 \end{propn}

\begin{propn} The following results are straightforward:

\begin{enumerate}

\item
\label{ctsimg}The continuous image of a Lindel\"of space $X$ is Lindel\"of.

\item
Quotient spaces of Lindel\"of spaces are Lindel\"of.

\item
 Countable disjoint sums of Lindel\"of spaces are Lindel\"of.

\item
If $X$ is a countable union of Lindel\"of spaces then $X$ is Lindel\"of.

\end{enumerate}
\end{propn}

\subsection{The Lindel\"of Property and Separation Axioms}

There are some basic relations between separation axioms in a \Lin topological space.
 We have the following very interesting theorem:
\begin{thm} 
\label{regular+L=normal}
Every regular Lindel\"of space $X$ is normal.
\end{thm} 
 However, even though in Lindel\"of spaces regularity entails normality, in such spaces we have that Hausdorffness does not entail regularity, as the following example shows.

\begin{xmpl} There exists a countable, Hausdorff not regular topogical space $X$ (hence, $X$ is a Lindel\"of, Hausdorff not regular topological space). 
\end{xmpl}
 This is the well-known Irrational Slope topology (example 75,  \cite{SS}). 
We shall provide detailed proofs of all properties of this example, slightly restructure the proof that it is  not regular, and provide illustrations for the core steps.

\begin{construction}

Let $X= \{(x,y) : y \geq 0, x,y \in \Q\}$; in other words, 
$X$ consists of all points with both coordinates rational in the closed upper half-plane of $\R ^2$. 
Since $X$ is a product of two countable sets (two copies of $\Q$), $X$ is countable, and hence Lindel\"of in any topology on $X$. 

From now on, to avoid ambiguity in notation, we write
$
\langle a,b \rangle
$
for $(a,b)\times\{0\}$, where $(a,b)$ is an open interval in $\R$, and we identify $\Q\times\{0\}$ with $\Q$.

Now, we topologise $X$ as follows.

Let $(x,y) \in X$ and $\epsilon > 0$. 
Define an $\epsilon$-neighborhood of $(x,y)$ as follows:
\begin{displaymath}
N_\epsilon (x,y) = \{(x,y)\} 
\cup \left( \langle x+ y \sqrt{2} - \epsilon, x+ y \sqrt{2} + \epsilon \rangle \cap \Q \right) 
\cup \left( \langle x-y\sqrt{2} - \e, x-y\sqrt{2} + \e\rangle \cap \Q \right),
\end{displaymath}
in other words, each $N_\e (x,y)$ consists of $\{(x,y)\}$ plus two intervals on the rational $x$-axis centred at the two irrational points $x \pm y\sqrt{2}$. 
Denote $\langle x\pm y\sqrt{2} - \e, x\pm y\sqrt{2} + \e\rangle \cap \Q$ by $ B_\e (x\pm y\sqrt{2})$. 
Note that the lines joining $(x,y)$ with these points have slopes $\pm \frac{\sqrt{2}}{2}$, respectively. 

\input{IrratSlope1pic.tex}

Note that if $y=0$ (in other words, 
 $(x,0) \in X$) then its neighborhood consists of rational points in $\langle x-\e, x+\e\rangle$; 
in other words 
 $N_\e (x,0) = B_\e (x) (=\langle x-\e, x+\e\rangle \cap \Q)$.

Define a topology $\tau$ on $X$ as follows:

$U \subset X$ is open if and only if 
 for every point $(x,y) \in U$ there is $\e > 0$ such that $N_\e (x,y) \subset U$.
This is indeed a topology on $X$, since 
$
\B = \{N_\e (x,y) : 
 \e>0\}
$
is a system of local neighborhood bases 
for each $(x,y) \in X$, i.e.
\begin{enumerate}
\item
$\bigcup \B = X$,
\item
For each $(x,y) \in X$ and $\e_1, \ \e_2>0$, $N_{\e_1} (x,y) \cap N_{\e_2} (x,y) = N_{\min\{\e_1,\e_2\}} (x,y)$,
\item
Let $(x,y) \in N_{\e_1} (x_1,y_1) \cap N_{\e_2} (x_2,y_2)$ (we assume $(x_1,y_1)\neq (x_2,y_2)$). 
Then the only possibility for $(x,y)$ is $y=0$, 
since if $(x_1,y_1) \neq (x_2,y_2)$ then $N_{\e_1} (x_1,y_1) \cap N_{\e_2} (x_2,y_2) \subset \Q$.

\includegraphics[width=13cm, trim=0cm 7cm 0cm 0cm, clip=true]{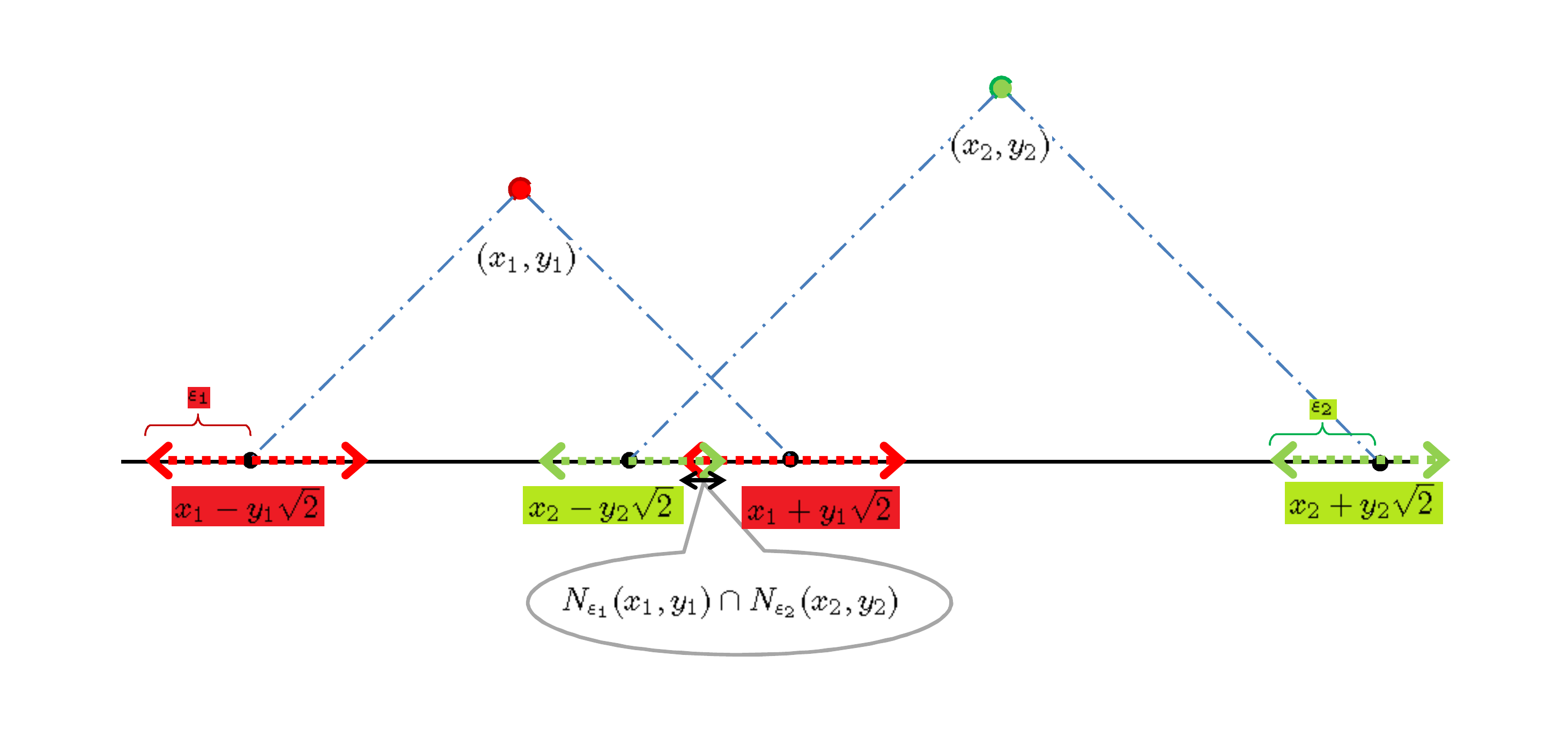}

That means that $(x,0)$ is in a rational open interval $\langle x-\e,x+\e\rangle \cap \Q = B_\e (x,y)$, where 
\begin{displaymath}
B_\e (x,y) \subset 
[ B_{\e_1} (x_1-y_1 \sqrt{2})\cup B_{\e_1} (x_1+y_1 \sqrt{2})]
\cap[B_{\e_2} (x_2-y_2\sqrt{2} ) \cup B_{\e_2} (x_2+y_2\sqrt{2} ) ]
\subset N_{\e_1} (x_1, y_1 ) \cap N_{\e_2} (x_2,y_2)
\end{displaymath}
\end{enumerate}
\end{construction}

\begin{cl}X is Hausdorff.
\end{cl}

\begin{proof}
Let $(x_1,y_1),(x_2,y_2)$ be two distinct points in $X$.
We have three cases:

\begin{case}[1]
 $(x_1,0) \neq (x_2,0)$

So we have two distinct points in $\Q$. Since $\Q$ is Hausdorff in its relative Euclidean topology, there is $\e > 0$ such that 
\begin{displaymath}
B_\e (x_1) \cap B_\e (x_2) = (\langle x_1 - \e, x_1 + \e\rangle \cap \Q) \cap (\langle x_2 - \e, x_2 +\e \rangle \cap \Q) = \emptyset.
\end{displaymath}
\end{case}
\begin{case}[2]
 $(x_1,0) \neq (x,y), \ y>0$.

\includegraphics[width=13cm, trim=0cm 7cm 0cm 0cm, clip=true]{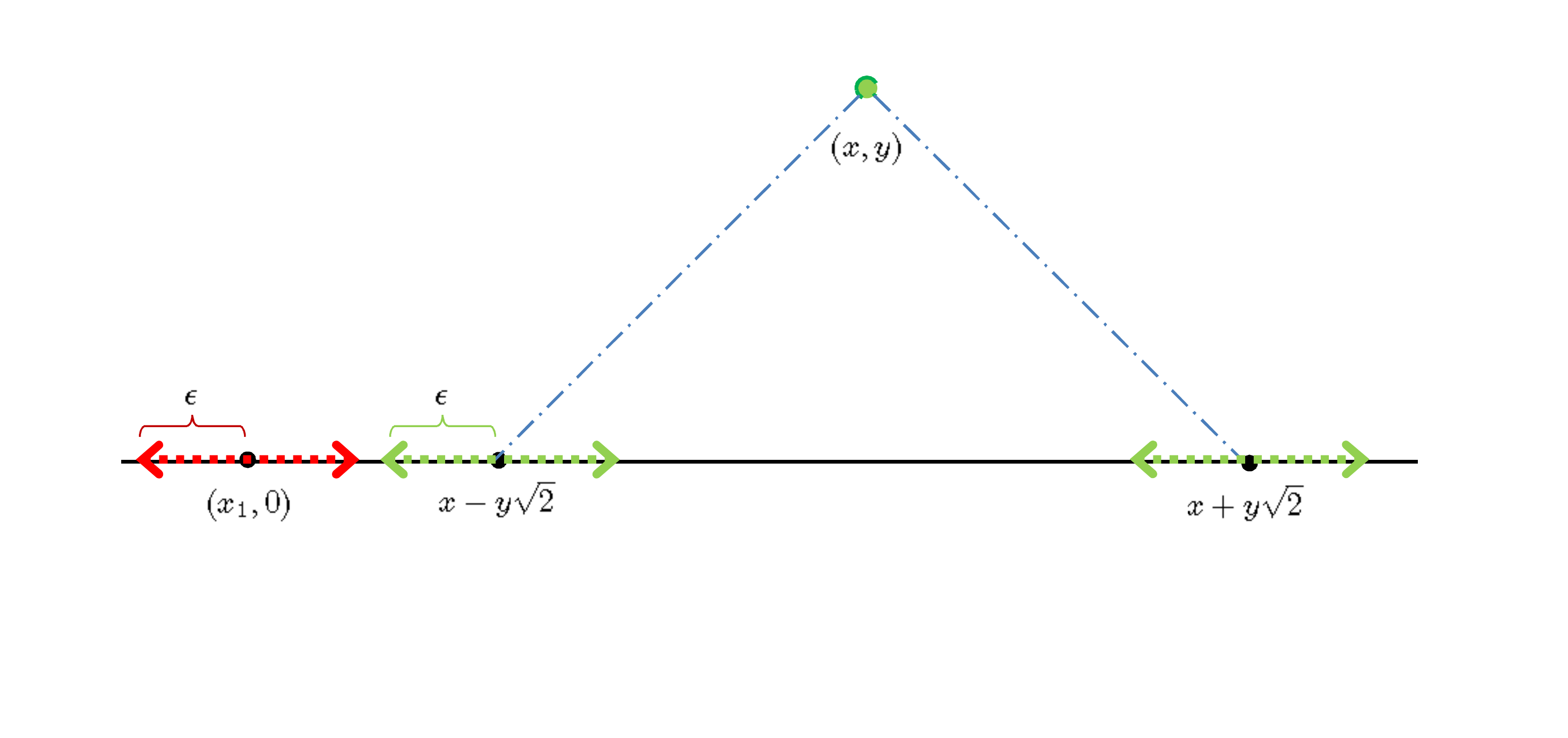}

Since $x_1$ is rational and $x\pm \sqrt{2}$ are irrationals, $x_1 \neq x \pm \sqrt{2}$. 
Because $\R$ is Hausdorff there is $\e > 0 $ such that 
\begin{displaymath}
\langle x-y\sqrt{2} - \e , x-y\sqrt{2} + \e\rangle \cap \langle x_1 -\e, x_1 +\e\rangle \cap \langle x+y\sqrt{2} - \e,x+y\sqrt{2} + \e\rangle = \emptyset.
\end{displaymath}
Then
$
N_\e (x_1,0) \cap N_\e (x,y) = \emptyset.
$

\end{case}
\begin{case}[3]
 $(x_1,y_1) \neq (x_2,y_2)$, $y_1 > 0, \ y_2 > 0$.

Recall that neighborhoods of $(x_1,y_1)$ and $(x_2,y_2)$ may have common points only in $\Q \subset \R$. 
Then $x_1\pm y_1\sqrt{2}, \ x_2 \pm y_2\sqrt{2}$ are four different points in $\R$:

Indeed, suppose (without loss of generality) that
$x_1+y_1\2=x_2-y_2\2$.

\includegraphics[width=13cm, trim=0cm 7cm 0cm 0cm, clip=true]{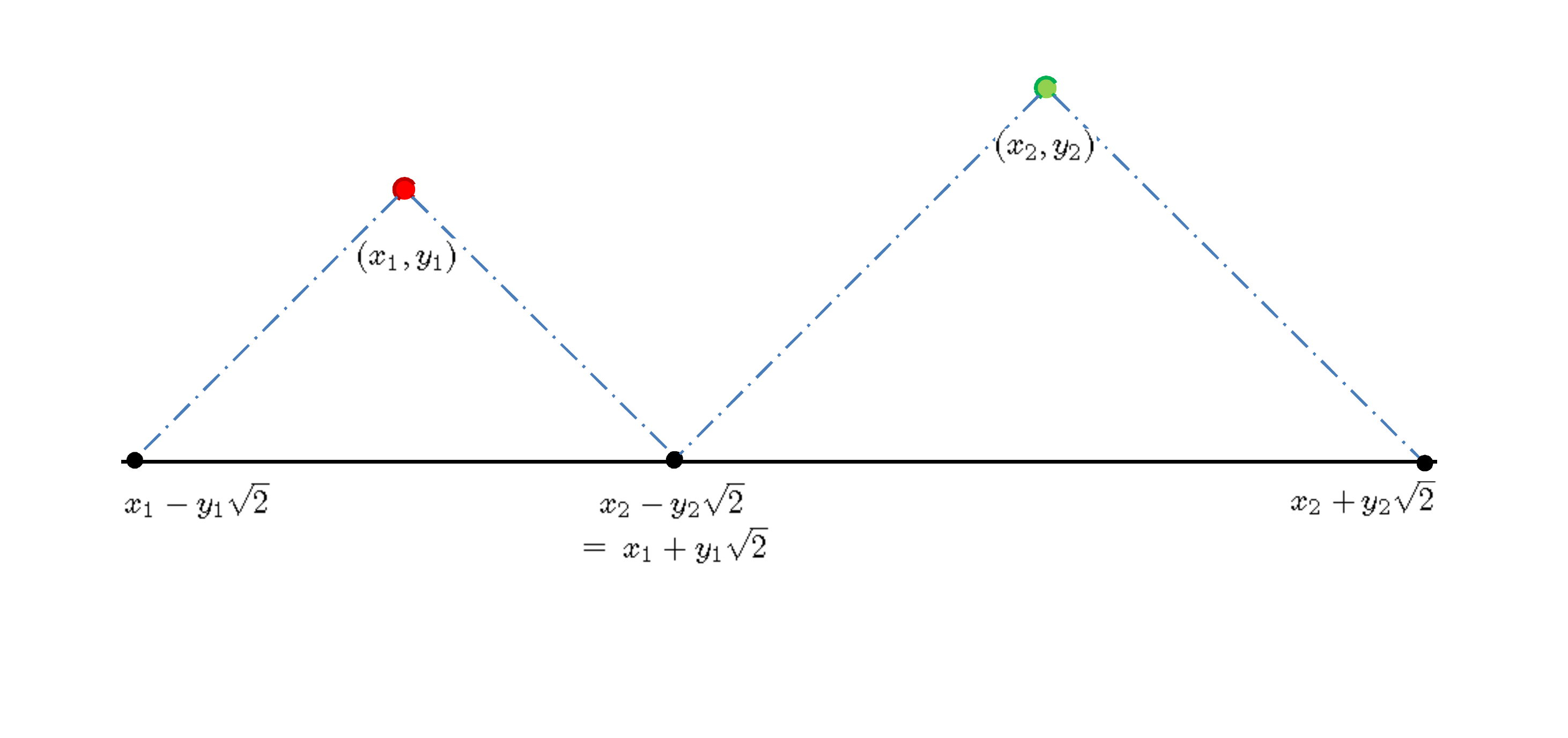}

Then $(y_1+y_2)\sqrt{2} = x_2-x_1$.
As $y_1,\ y_2 \in \Q$, the left hand side is in $\R \setminus \Q$. 
The right-hand side is in $\Q$ - contradiction. 
Since $\R$ is Hausdorff, there is $\e > 0 $ such that $\e$-neighborhoods of $x_i\pm y_i\sqrt{2}$ do not intersect.
Then $N_\e (x_1,y_1) \cap N_\e (x_2,y_2) = \emptyset$.
\end{case}
Hence, $X$ is Hausdorff. 
\end{proof}

\begin{cl}X is not regular.
\end{cl}
\begin{proof}
In order to see that $X$ is not regular, let us look at the closure of a basic neighborhood - $\bar{N_{\e_0} (x_0,y_0)}$, for an arbitrary $(x_0,y_0) \in X$ and $\e>0$.

Let us first define the following sets (strips in $\R^2 _+$), in the case when $y_0 > 0$. 

\begin{align}
\nonumber
S_{0} &=\{ (x,0) \in X : x \in B_{\e_0} (x_0-y_0\2) \cup B_{\e_0}(x_0+y_o\2) \}, \\
\nonumber
S_{1} &= \{(x,y) \in X : y>0,\ |x_0+y_0\sqrt{2} - (x-y\2)|<\e_0\},\\
\nonumber
S_{2} &= \{(x,y) \in X : y>0, \ |x_0-y_0\2 - (x-y\2)|< \e_0\},\\
\nonumber
S_{3} &= \{(x,y) \in X : y>0,\ |x_0+y_0\sqrt{2} - (x+y\2)|<\e_0\},\\
\nonumber
S_{4} &=\{(x,y) \in X : y>0, \ |x_0-y_0\2 - (x+y\2)|< \e_0\}. 
\end{align}

For simplicity the picture will show $S_i$ for a fixed point $(x,y)>0,\ \e_0>0$.

\includegraphics[width=13cm]{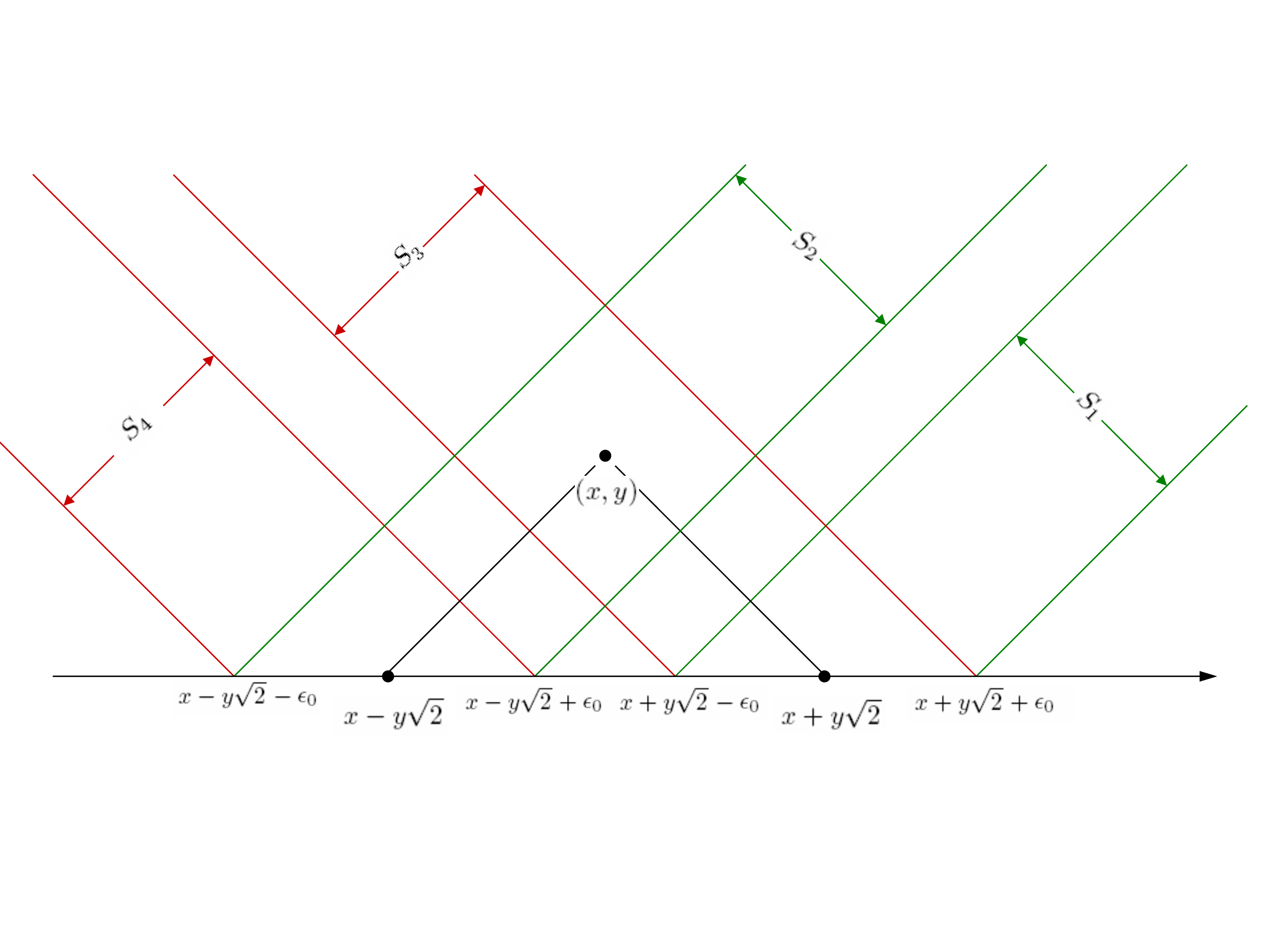}

We shall prove that 
\begin{displaymath}
\bar{N_{\e_0} (x_0,y_0)} = S_0\cup S_1\cup S_2\cup S_3\cup S_4.
\end{displaymath}

Let us show that $S_0\cup S_1\cup S_2\cup S_3\cup S_4 \subseteq \bar{N_{\e_0} (x_0,y_0)} $.

Let $(x,0) \in S_0$ and without loss of generality let $(x,0) \in B_{\e_0} (x_0-y_0\2)$.

\includegraphics[width=13cm, trim=0cm 7cm 0cm 0cm, clip=true]{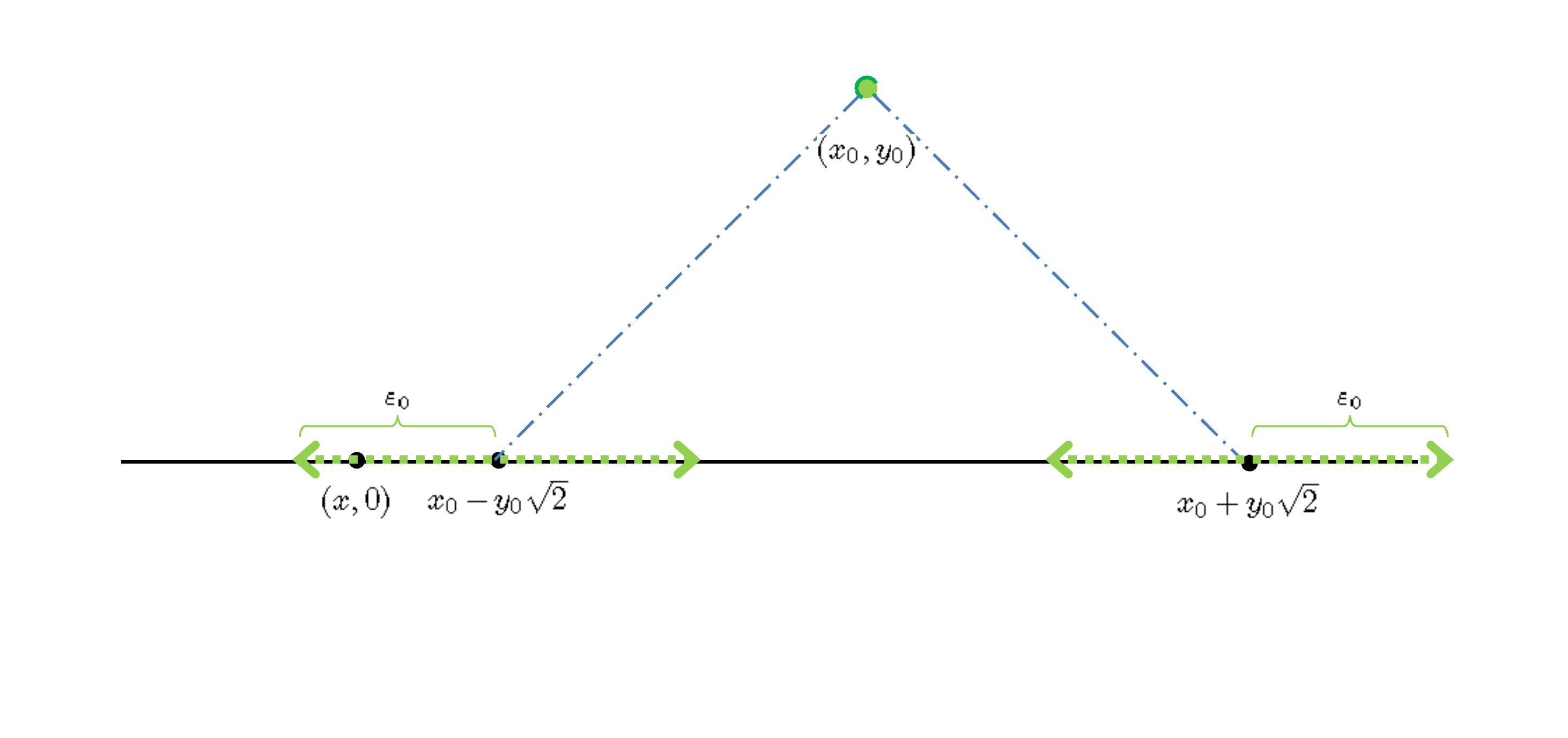}

Then for all $\e>0$, $N_\e (x,0) \cap N_{\e_0} (x_0,y_0) \neq \emptyset$, because $N_\e (x,0) =B_\e (x)$ and $B_\e(x) \cap B_{\e_0} (x_0-y_0\2) \neq \emptyset$ and $B_\e(x) \cap B_{\e_0} (x_0-y_0\2) \subset N_{\e_0} (x_0,y_0)$. 
Let $(x,y) \in S_1$: 

\includegraphics[width=13cm, trim=0cm 7cm 0cm 0cm, clip=true]{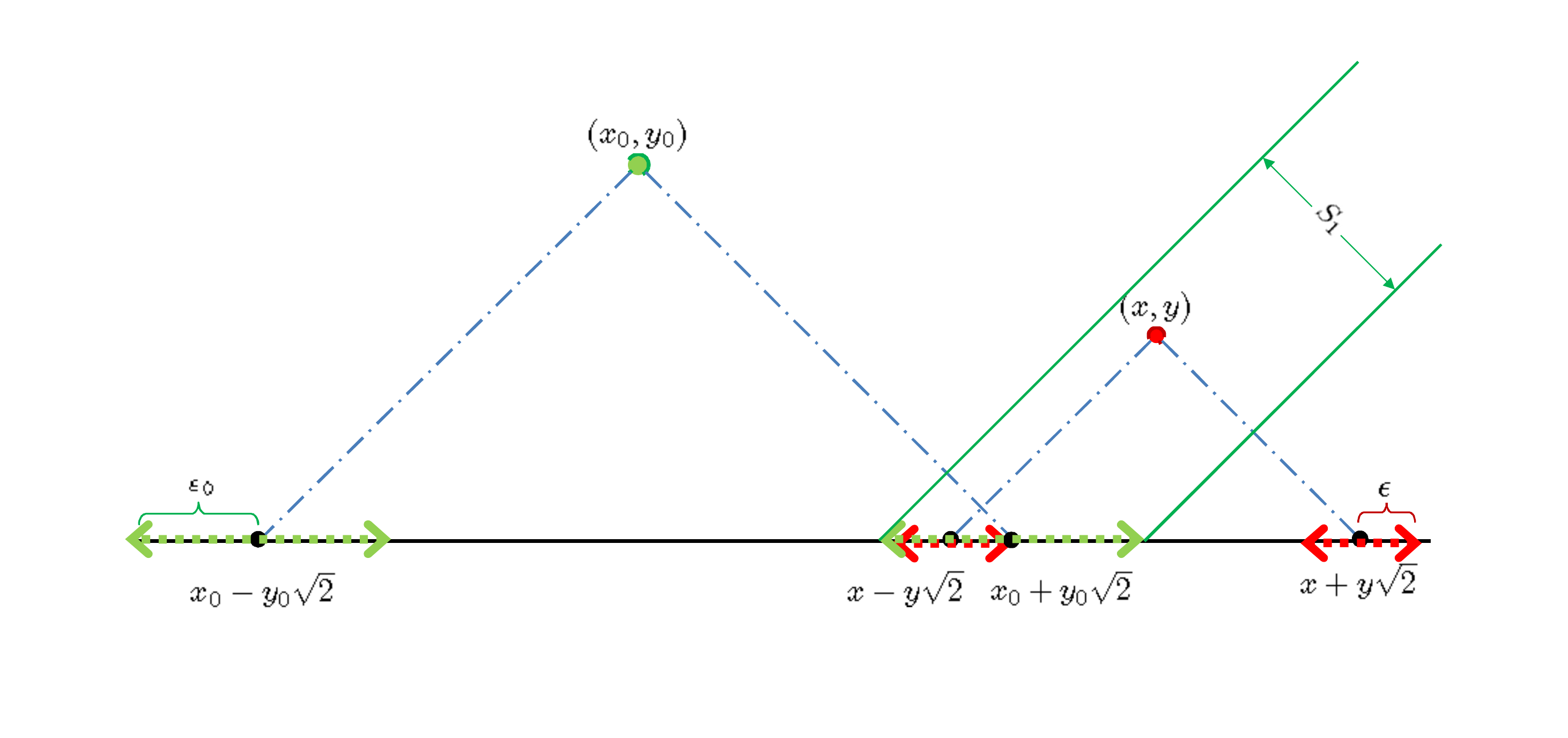}

$(x,y) \in S_1$, hence, for every $\e>0$ we have that $N_\e (x-y\2) \cap B_{\e_0} (x_0+y_0\2) \neq \emptyset$. 
Hence, for every $\e>0$ we have that $B_\e (x,y) \cap N_{\e_0} (x_0,y_0) \neq \emptyset$.
\\

Cases when $(x,y) \in S_2$ ($S_3, \ S_4$) are analogous.

Hence 
\begin{displaymath}
S_0\cup S_1\cup S_2\cup S_3\cup S_4 \subseteq \bar{N_{\e_0} (x_0,y_0)}.
\end{displaymath}

For the converse, let us prove that if $(x,y) \notin S_0\cup S_1\cup S_2\cup S_3\cup S_4 $, then there is $\e>0$ such that $N_\e (x,y) \cap B_{\e_0} (x_0,y_0) = \emptyset$, i.e.
$(x,y) \notin \bar{N_{\e_0} (x_0,y_0)}$. 

\includegraphics[width=13cm, trim=0cm 7cm 0cm 0cm, clip=true]{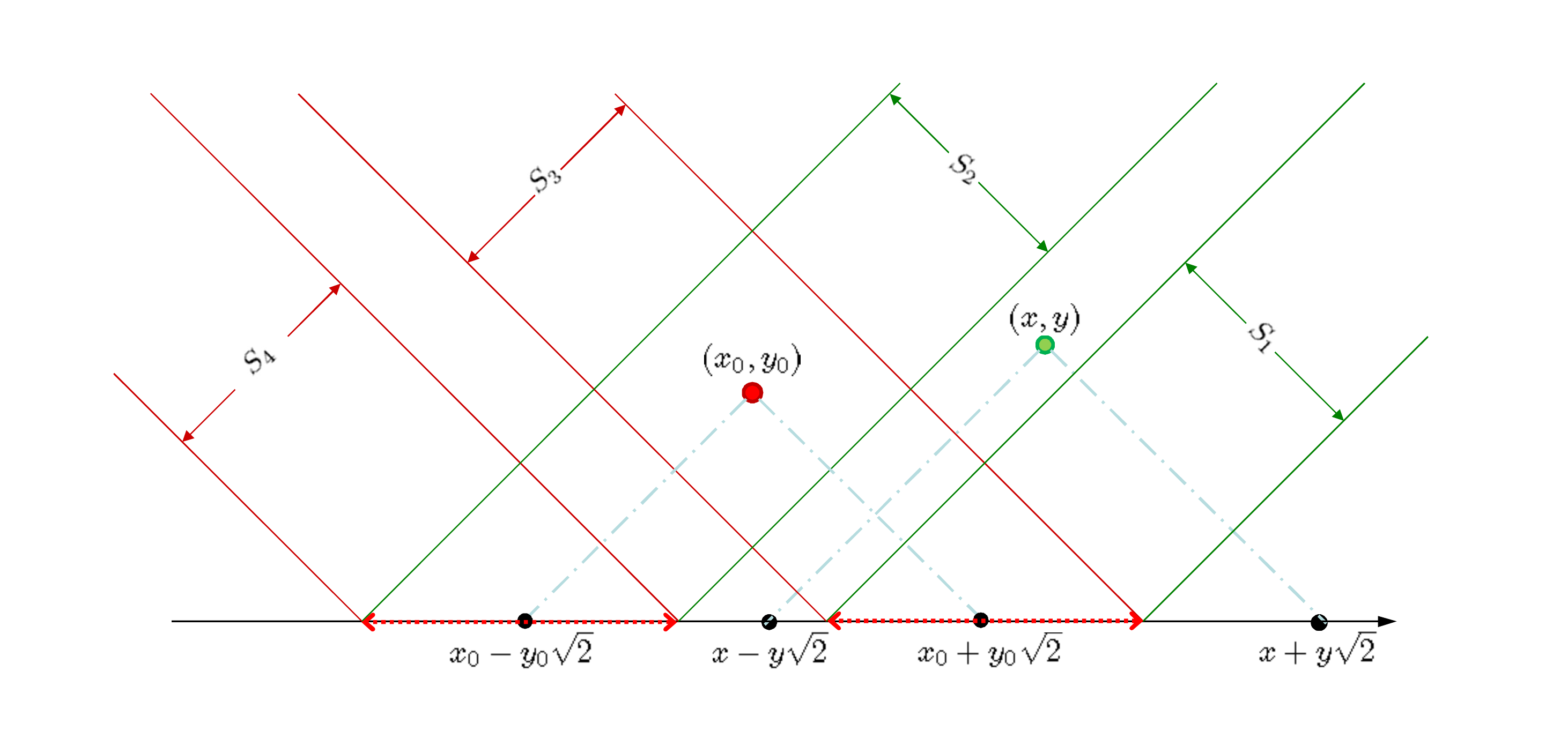}

Since $(x,y) \notin S_0\cup S_1\cup S_2\cup S_3\cup S_4$, we have that all four points $x \pm y\2$, $x_0\pm y_0\2$ are four different points in $\R$ and $\R$ is Hausdorff. 
So, we can find $\e>0$ such that 

\begin{displaymath}
\left( B_\e (x-y\2) \cup B_\e (x+y\2) \right) \cap \left( B_{\e_0} (x_0-y_0 \2) \cup B_{\e_0} (x_0+y_0 \2) \right) \neq \emptyset.
\end{displaymath}
Then $N_\e (x,y) \cap N_{\e_0} (x_0,y_0) = \emptyset$.
Hence $\bar{N_{\e_0} (x_0,y_0)} = S_0\cup S_1\cup S_2\cup S_3\cup S_4$ 
for every $\e_0 >0$ and $(x_0,y_0) \in X$, $y_0>0$.
\\

For $(x_0,0)$, $\bar{N_{\e_0} (x_0,0)} = P_1 \cup P_2 \cup B_{\e_0} (x_0)$, where

\begin{align}
\nonumber
P_1 &= \{(x,y) \in X : y>0,\ |x+y\sqrt{2} - x_0|<\e_0\}\\
\nonumber
P_{2} &= \{(x,y) \in X : y>0, \ |x-y\2 - x_0|< \e_0\}\\
\nonumber
\end{align}

\includegraphics[width=13cm, trim=0cm 7cm 0cm 0cm, clip=true]{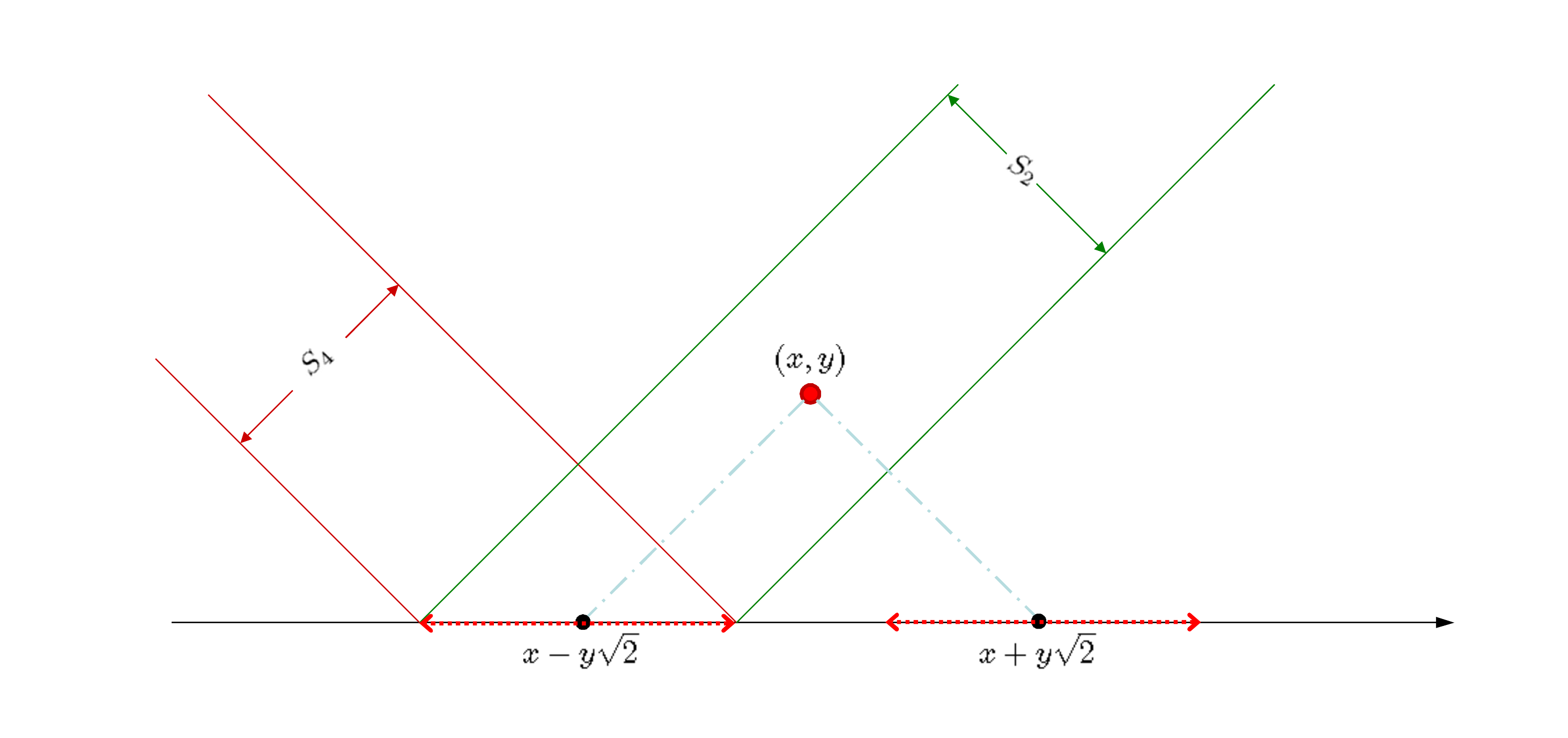}

A general remark about the rational values of $\e_0$: the two endpoints $x_0 \pm \e_0$ might also belong to $\bar{N_{\e_0} (x_0)}$, but no point from the lines in the upper plane lies on the lines originating at $x_0 \pm \e_0$ are in $\bar{N_{\e_0} (x_0)}$,  
  because those lines are with slope $\pm \frac{\2}{2}$, hence the $y$-coordinate of such points must be irrational, hence not in $X$. 

Hence (abusing the notation for $S_2,\ S_4$ as in the previous case, without using indices to show that $S_i$ depend on $(x_0,y_0)$ and $\e_0$), we can write that

\begin{displaymath}
\bar{N_{\e_0} (x_0)} = S_0\cup S_2\cup S_4,
\end{displaymath}
 when $x_0+\e_0$ is irrational, and 

\begin{displaymath}
\bar{N_{\e_0} (x_0)} = S_2 \cup S_4 \cup [x_0-\e_0, x_0+\e_0]\cap\Q
\end{displaymath}
for rational $\e_0$. 

The end points are irrelevant in some sense (for the purpose of proving that $X$ is not regular). 

We shall show that $X$ is not only not regular but in fact it is not a Urysohn space.

It suffices to prove that the closures of any two basic neighborhoods have a nonempty intersection. 

This is ``obvious" from the picture (of the closures):

\includegraphics[width=13cm, trim=0cm 7cm 0cm 0cm, clip=true]{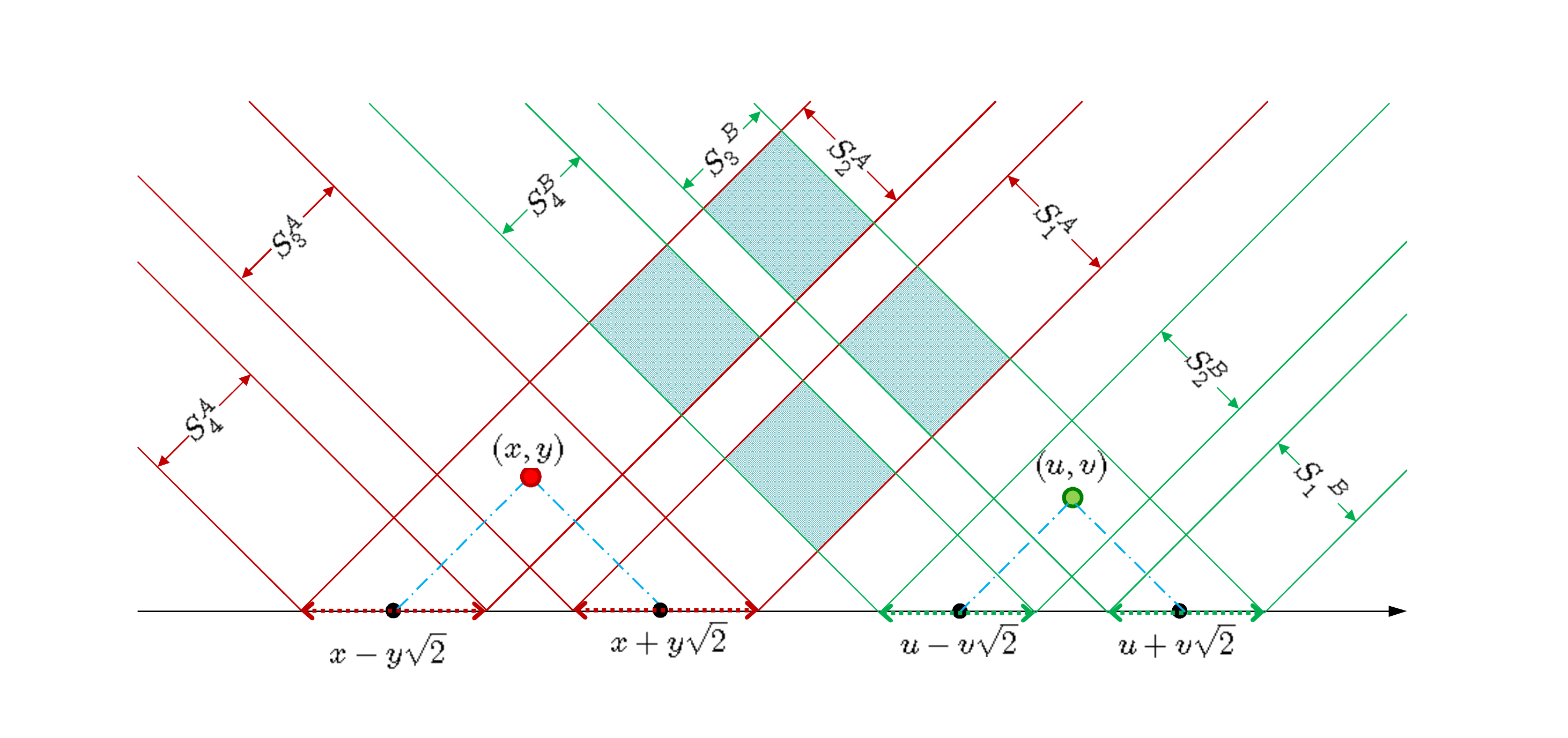}

(Note that ${S_i}^B, S_i ^A$ denote the respective "strips" for $A$ and $B$)

If, without loss of generality, $u>x$, then $({S_3}^B\cup S_4 ^B) \cap (S_1 ^A \cup S_2 ^A)\neq \emptyset$.
If two points have open neighborhoods with disjoint closures then these neighborhoods can be considered basic, hence if we prove that basic neighborhoods of any two points have non-disjoint closures, $X$ will not be Urysohn, and hence not regular.

For some $\e>0$, the points on the border lines of the strips $S_1,\ldots,S_4$ might also be in $\bar{N_{\e} (x,y)}$, 
but for the purpose of proving that $X$ is not Urysohn (regular), this is irrelevant. 

What is important is that the closures of basic neighborhoods of any two points in $X$ intersect.
\begin{center}
\includegraphics[width=10cm, trim=0cm 5cm 0cm 5cm, clip=true]{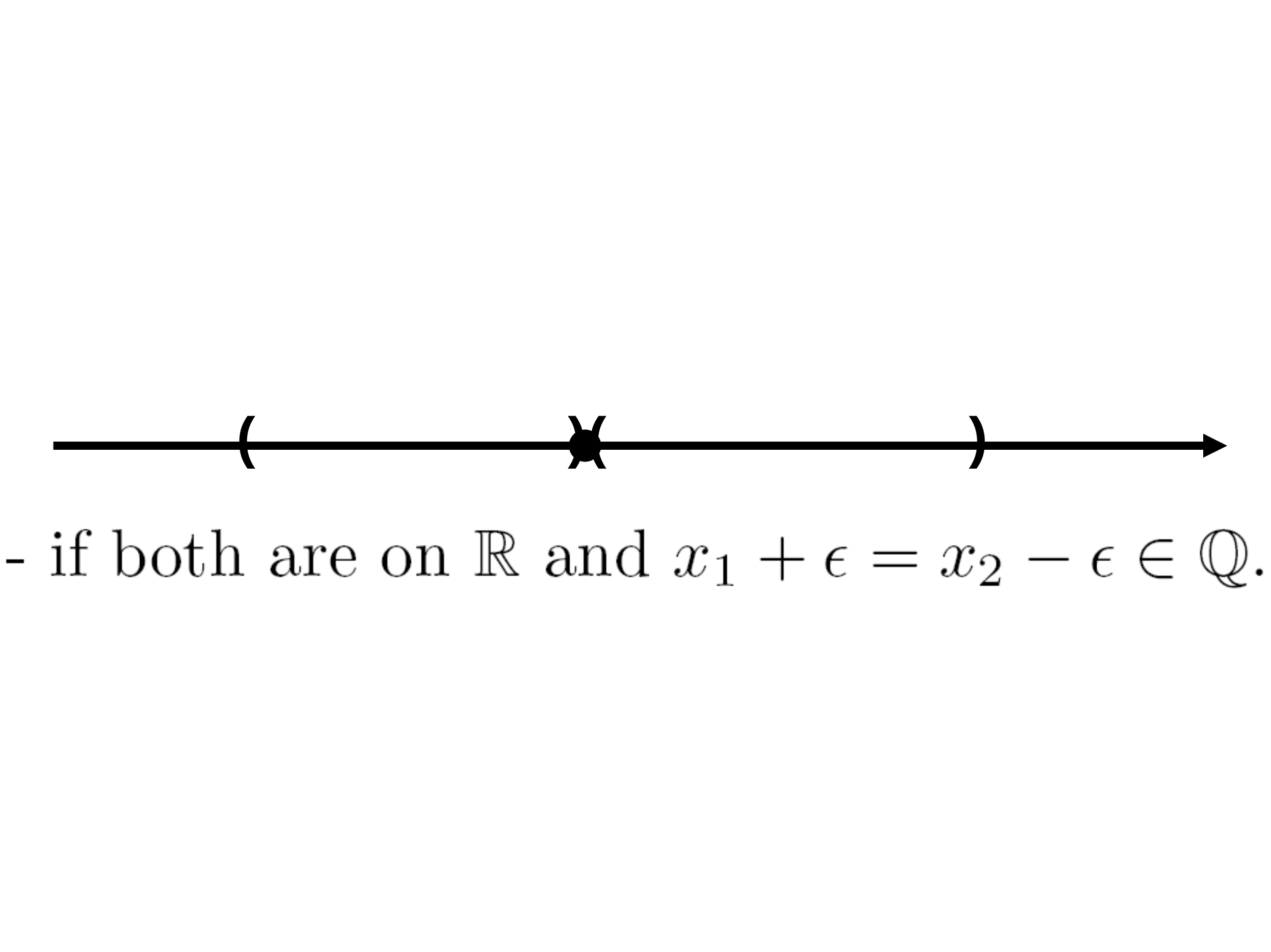}
\end{center}

In all other cases, if one of the points has $y$-coordinate $>0$, then the interior of the corresponding strips $S_1,\ldots,S_4$ will contain points from $X$, hence the interiors will intersect ($\Q \times \Q$ is dense in $\R^2$). 
\end{proof}

Another fundamental example of a \Lin space is the Sorgenfrey line.

\begin{xmpl}[Sorgenfrey line]
The Sorgenfrey line $\Sorg$ is $\R$ with the topology generated by the base  $\B=\{ [ a,b) : a<b\}$. 
\end{xmpl}

\begin{xmpl} \label{Sorg}$\Sorg$ is Lindel\"of and normal.
\end{xmpl}

\begin{proof} 
We will provide the proof  given in \cite{SxS}.

Let $\U$ be an open cover of $\Sorg$. 
We define $\V = \{\halo{U}^{\R}: U \in \U\}$, where $\halo{U}^{\R}$ is the interior of the set $U$ with respect to the Euclidean topology 
. 
By Example \ref{Lsmpl}(\ref{RLind}), there is a countable subfamily $\U_1 \subset \U$ with $\bigcup \{\halo{U}: U \in \U_1\}=\bigcup\V$. 

Define $Z = \R \setminus \bigcup \V$. 
We shall  show that $Z$ is countable. 
For every $x\in Z$ there is $U_x\in\U$ such that $x\in U_x$, and there is $b_x\in\R$ such that $x<b_x$ and $[x,b_x)\subseteq U_x$. 
For every $x \in Z$ let $\al_x \in \Q$ be such that $x<a_x<b_x$. 
Define a mapping $f: Z \rightarrow \Q$ by $f(x) = a_x$, and let us show that it is a one-to-one function from $Z$ into $\Q$. 
Suppose $x\neq y \in Z$.
Without loss of generality, we may assume that $x<y$. 
If $x<y<b_x$ then $y \in (x,b_x) \subseteq \halo{U}^{\R}$, for some $U\in \U$ with $[x,b_x)\subset U$. 
But this means that $y \in \halo{U}^\R \in \V$ - a contradiction to $y \in Z = \R \setminus \bigcup \V$. 
Therefore $\al_x <b_x\leq y<\al_y$. 
Hence, $Z$ is countable, and so is $\U_0 = \{U_x : x\in Z\}$. 

Then $\U_0 \cup \U_1$ is a countable subcover of $\U$.

We  shall show that $\Sorg$ is regular, and hence by Theorem \ref{regular+L=normal}, it is normal. 
Let us point out that any basic open set $[a,b)$ is also closed,
 because $\displaystyle{X\setminus[a,b)=\bigcup_{n\in\N}[a-2n,a-n)\cup\bigcup_{n\in\N}[b+n,b+n+1)}$. 
Hence, whenever $U$ is open and $x\in U$, there exists an open set $V$ ($=[a,b)$) such that $x\in V\subset\bar{V}\subset U$.
Hence, $\Sorg$ is regular.
\end{proof}

However, the product of even two Lindel\"of spaces need not be Lindel\"of.
In \cite{Sorg}
Sorgenfrey proved that $\Sorg\times\Sorg$ is not Lindel\"of by proving that it is not normal. 
We shall give a different proof by using the following proposition:

\begin{propn}
\label{closed discrete subsets of HL}
If $X$ is  Lindel\"of and $F \subset X$ is closed and discrete then $F$ is countable.
\end{propn}

\begin{xmpl}\label{SorgplanenotL} $\Sorg \times \Sorg$ is not Lindel\"of. 
\end{xmpl} 

\begin{proof}
Let  us consider $D = \{(x, -x) : x \in \Sorg\}\subset\Sorg \times \Sorg$.
It is uncountable, as it is equicardinal to $\R$. 
We shall show that $D$ is closed and discrete in $\Sorg \times \Sorg$, and hence, by Proposition \ref{closed discrete subsets of HL},  $\Sorg \times \Sorg$ is not Lindel\"of. 

First we show that $D$ is closed by showing that $(\Sorg \times \Sorg)\setminus D$ is open in $\Sorg \times \Sorg$. 
Let $D^+ = \{(x,y) : x+y > 0\}$ and $D^- = \{(x,y) : x+y < 0\}$.
Then $(\Sorg \times \Sorg) \setminus D = D^+ \cup D^-$.

Indeed, let $(x,y) \in D^+$.
Then every basic neighborhood of $(x,y)$ does not intersect $D$ (because it is of the form $[ x, x+ \epsilon) \times [y, y+ \epsilon)$.
If $(x,y) \in D^-$ then the neighborhood $[x,\frac{x-y}{2}) \times [y, \frac{y-x}{2})$ does not intersect $D$. 
The sets $D^+$ and $D^-$ are both open in $\Sorg\times\Sorg$, hence, $D$ is closed.
\begin{center}
\includegraphics[width=10cm]{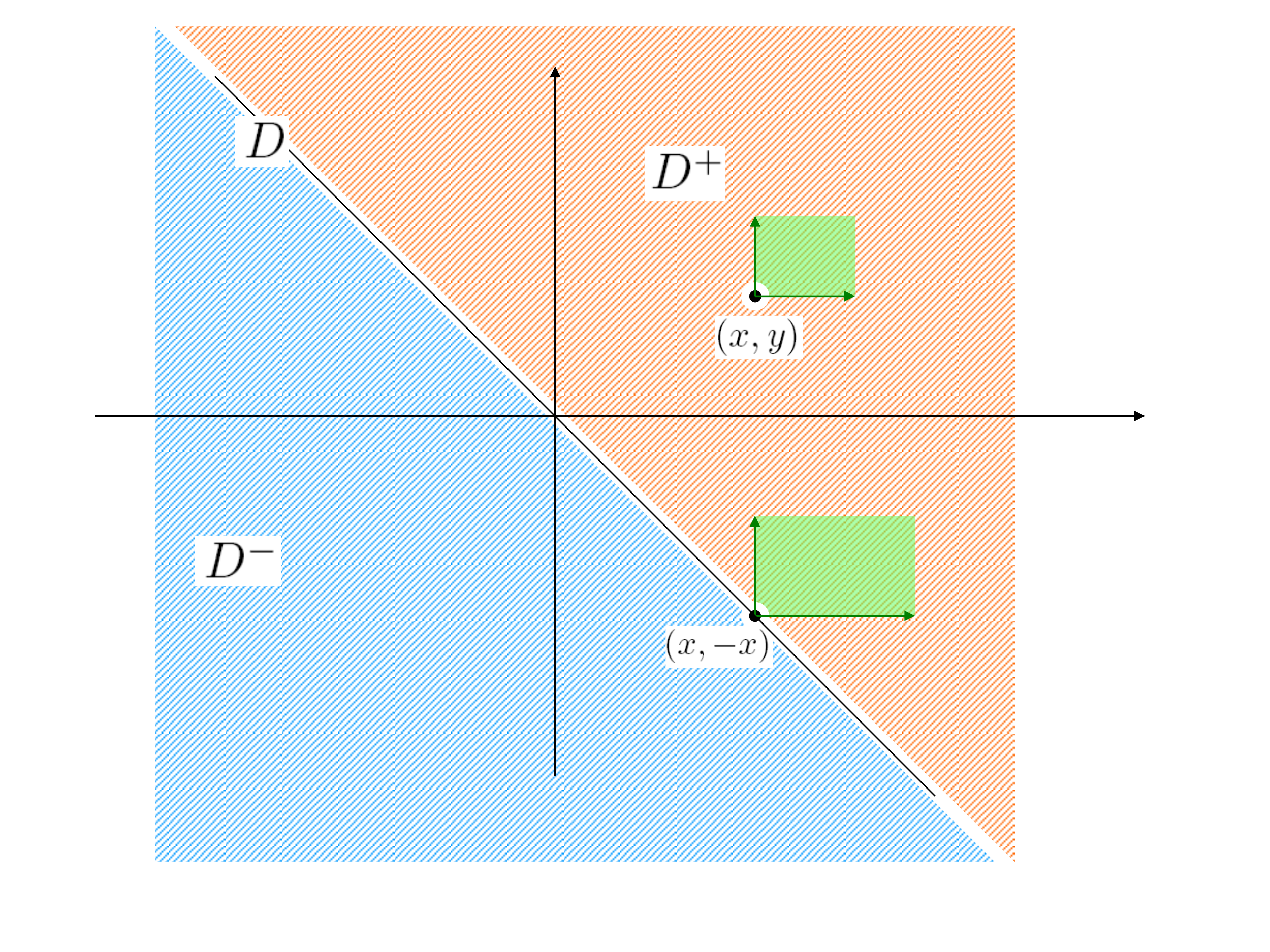}
\end{center}

$D$ is discrete because for every $(x,-x) \in D$ we have that each basic neighborhood of $(x,-x)$ in $\Sorg \times \Sorg$ (which is of the form $[x,x+\epsilon) \times [-x, -x+\epsilon)$, for $\epsilon > 0$) intersects $D$ in just one point, $(x,-x)$. 
 This completes the proof. 
\end{proof}

\subsection{Lindel\"ofness and Other Properties}

Here we consider the interrelations between \Lind ness and other topological properties, such as separability, second countability, and CCC.

\begin{thm} \cite{SDb} \label{metric} Let $X$ be a metric space. Then the following are equivalent:

\begin{enumerate}
\item
$X$ is Lindel\"of,
\item
 $X$ is hereditarily Lindel\"of,
\item
$X$ is second countable,
\item
$X$ is separable,
\item
$X$ satisfies the countable chain condition.
\end{enumerate}
\end{thm}

The following example shows that in general, separability does not follow from being CCC and \Lind.

\begin{xmpl} \cite{Win} There is a Hausdorff, Lindel\"of, first countable space $X$ which is CCC but not separable.
\end{xmpl} 

For more clarity, we will slightly modify the construction of Winkler, and we will provide visualizations of the main steps of the proof.

\begin{construction}

We will first construct a family of ``copies" of $\Q$ of length $\w$, consisting of pairwise disjoint (or else coinciding) sets. 

Let $\Q_a = \Q + a=\{q+a: q\in\Q\}$, for every $a \in \I$. 
 Since the translation map $f : \R \rightarrow \R$, given by $f(x) = x+a$ is a homeomorphism, for all $a \in \R$, $\Q_a$ are dense in $\R$. 

In $\{\Q_a:a\in\I\}$ there is an uncountable disjoint subfamily $D$.

Then $D$ consists of countable dense subsets of $\R$.

We can take a subfamily $\cZ \subset D$ enumerated by $\w_1$, say
$\cZ = \{\Q_{a_\al} : \Q_{a_\al}\in D, \al < \w_1\}$. 
From now on, we abuse notation and write $\Q_\al$ for $\Q_{a_\al}$, where $\al \in \w_1$ and $a \in \I$. 

Define $X = \{(x,\al) : x \in \Q_\al, \al<\w_1\}\subseteq \R\times\w_1$. 
\\
\\
\begin{center}
\includegraphics[width=10cm]{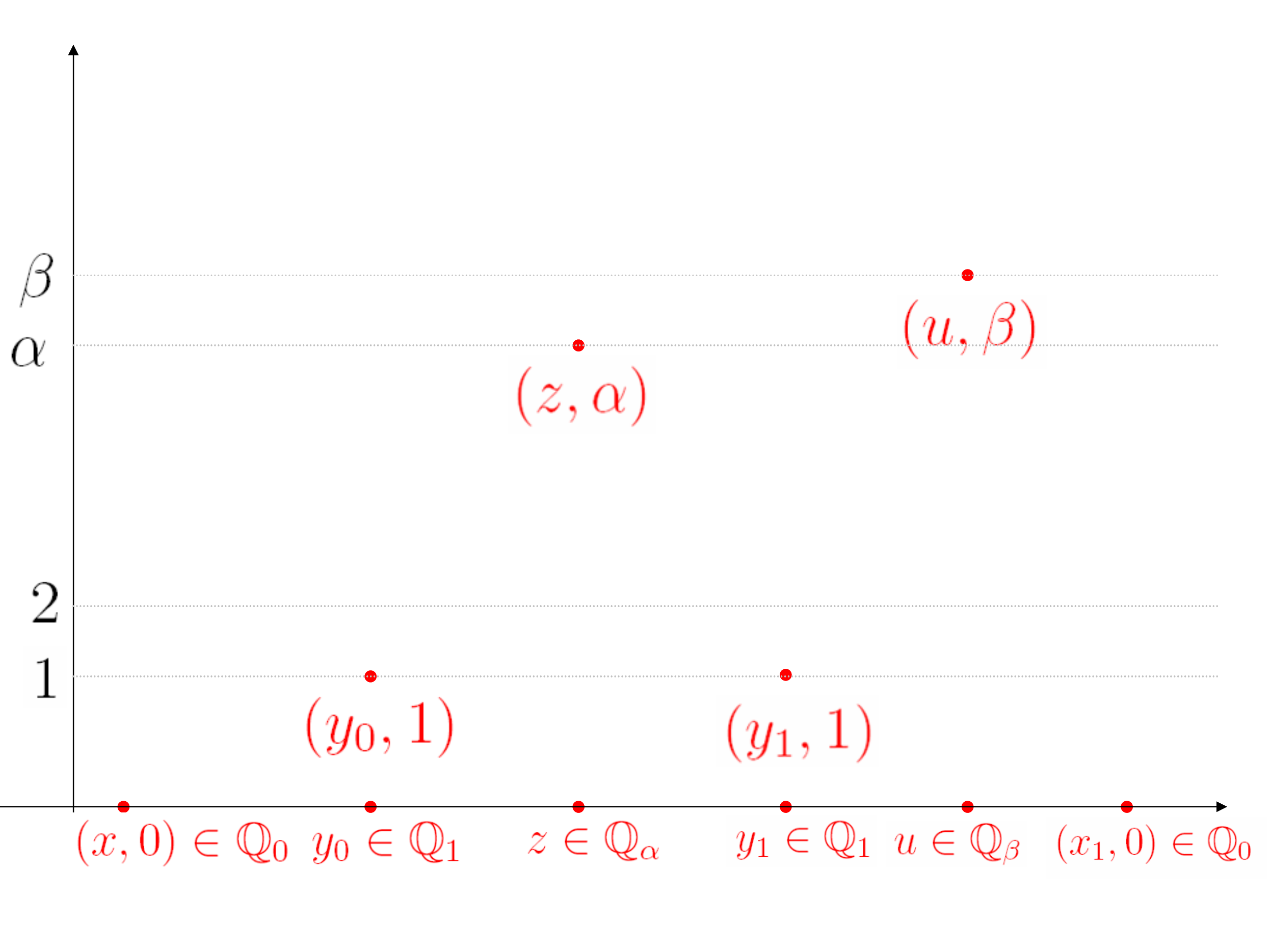}
\\
\end{center}

Let $(x, \al) \in X$ and $\epsilon > 0$. 
Define $O_\epsilon (x,\al) = \{(y, \be) \in X : |x-y| < \epsilon , \al \leq \be < \w_1\}$. 

Let us illustrate with some examples. 

$O_\e (x,0) = \{(y, \be) \in X : y \in \Q_\be, \ |x-y| < \epsilon , 0 \leq \be < \w_1\}$. Note that $x \in \Q_0$.

$O_\e (x,1) = \{(y, \be) \in X : y \in \Q_\be, \ |x-y| < \epsilon , 1 \leq \be < \w_1\}$. 
\\

\input{Winkler3pic.tex}
\\

We now show that $\B = \{O_\epsilon (x,\al) : (x,\al) \in X, \epsilon > 0\}$ is a base for a topology $\tau$ on $X$. 

We obviously have that $\cup \B = X$. 

Let $(x,\al) \in O_{\e_1} (x_1,\al_1) \cap O_{\e_2} (x_2,\al_2)$.
We have to find $\e>0$ such that $O_\e(x,\al) \subset O_{\e_1} (x_1,\al_1) \cap O_{\e_2} (x_2,\al_2)$.
Since $(x,\al) \in  O_{\e_i} (x_i,\al_i)$ it follows that $\al \geq \al_i$, $i = 1,2$. 
Without loss of generality suppose $x_1<x_2$ and $x_1-\e_1 < x_2-\e_2$ and $x_1+\e_1 < x_2-\e_2$. 
We have the following picture:
\\
\\

\input{Winkler4pic.tex}
\\
\\

Then we have that  $ x_2-\e_2<x<x_1+\e_1 $.
Let $\e>0$ be such that $x_2-\e_2<x-\e<x<x+\e<x_1+\e_1 $.
Then $O_\e(x,\al) \subset  O_{\e_1} (x_1,\al_1) \cap O_{\e_2} (x_2,\al_2)$.
Indeed, let $(y,\be) \in O_\e (x,\al)$.
Then $\be \geq \al \geq \al_i$, $i=1,2$.
Also $x-\e < y < x+\e$, hence $x_2-\e_2<y<x_1+\e_1$, i.e. $|x_1-y|< \e_1$ and $|x_2-y|< \e_2$. 
Hence, $(y,\be) \in  O_{\e_1} (x_1,\al_1) \cap O_{\e_2} (x_2,\al_2)$. 
Therefore, $\B$ is a base for a topology on $X$. 
\end{construction}

\begin{cl}X is Hausdorff.
\end{cl}
\begin{proof}

Note that $x_1=x_2$ implies $\al_1=\al_2$.
So, if we want $(x_1,\al_1) \neq (x_2,\al_2)$, then we just need $x_1\neq x_2$.

Let $(x_1,\al_1) \neq (x_2,\al_2)$ be two points in   $X$,
i.e. $x_1 \in \Q_{\al_1}$, $x_2 \in \Q_{\al_2}$.
Since $\Q_{\al_1} \cap \Q_{\al_2} = \emptyset$ we have that $x_1 \neq x_2$ and without loss of generality $x_2 < x_2$, $\al_1 < \al_2$

Let $\epsilon = \frac{|x_1 - x_2|}{2}$. 
Then $O_\epsilon (x_1, \al_1) \cap O_\epsilon (x_2, \al_2) = \emptyset$, because their projections in $\R$ are disjoint.
\end{proof}

\begin{cl}X is first countable.
\end{cl}

\begin{proof}
$X$ is first countable because $\{O_\frac{1}{n} (x,\al) : n \in \N\}$ is a countable local base for every $(x,\al) \in X$. 

Indeed if $U\subset X$ is open, then there is an $\e>0$ such that $O_\e (x,\al) \subset U$.
Take $n \in N$ such that $\frac{1}{n} < \e$.
Then $O_\frac{1}{n} (x,\al) \subset O_\e (x,\al) \subset U$. 
\end{proof}

\emph{Change of notation}

Let $O(a,b,\al) = \{(x,\be) \in X : a<x<b , \al\leq\be<\w_1\}$.
Then $O(a,b,\al)$ is open in $\tau$ and $\B_1 = \{O(a,b,\al) : a,b \in \R, \al<\w_1\}$ is again a base for the same topology $\tau$ of $X$. 

The proof is exactly the same as the one showing that open intervals form a base for the Euclidean topology on $\R$

\begin{cl} X is Lindel\"of.
\end{cl}

\begin{proof}
We will use the fact that X is Lindel\"of iff from any cover with basic open sets we can find a countable subcover. 
Let $\g = \{O(a_k,b_k,\al_k) : k \in K, \al_k < \w_1\}$ be an open cover of $X$ with basic open sets.
Let us note that $\al = 0$ must be among those $\{ \al_k : \al_k < \w_1\}$, otherwise points from $X \cap (\R \times \{0\})$ will not be covered. 
Let $A=\bigcup_{\al\in\w_1}\Q_\al$; note that $A$ is uncountable.
Then $A \subset \R$ and since $\R$ is metric with a countable base, so is $A$; hence, $A$ is Lindel\"of. 
Then $\U = \{(a_k,b_k) \cap A : k \in K\}$ is an open cover of $A$ in the Euclidean topology and we can choose a countable 
subcover which 
 we shall denote by $\Upsilon = \{(a_n, b_n) \cap A : n \in \N\} $.
Then $\{\al_n : n \in \N\}$ is a countable set of countable ordinals, hence $\al_0 = \sup \{\al_n: n \in \N\}$ is at most countable. 

Let us show that $\{O(a_n,b_n,\al_n) : n \in \N\}$ covers all but countably many elements of $X$. 

\begin{center}
\includegraphics[width=10cm]{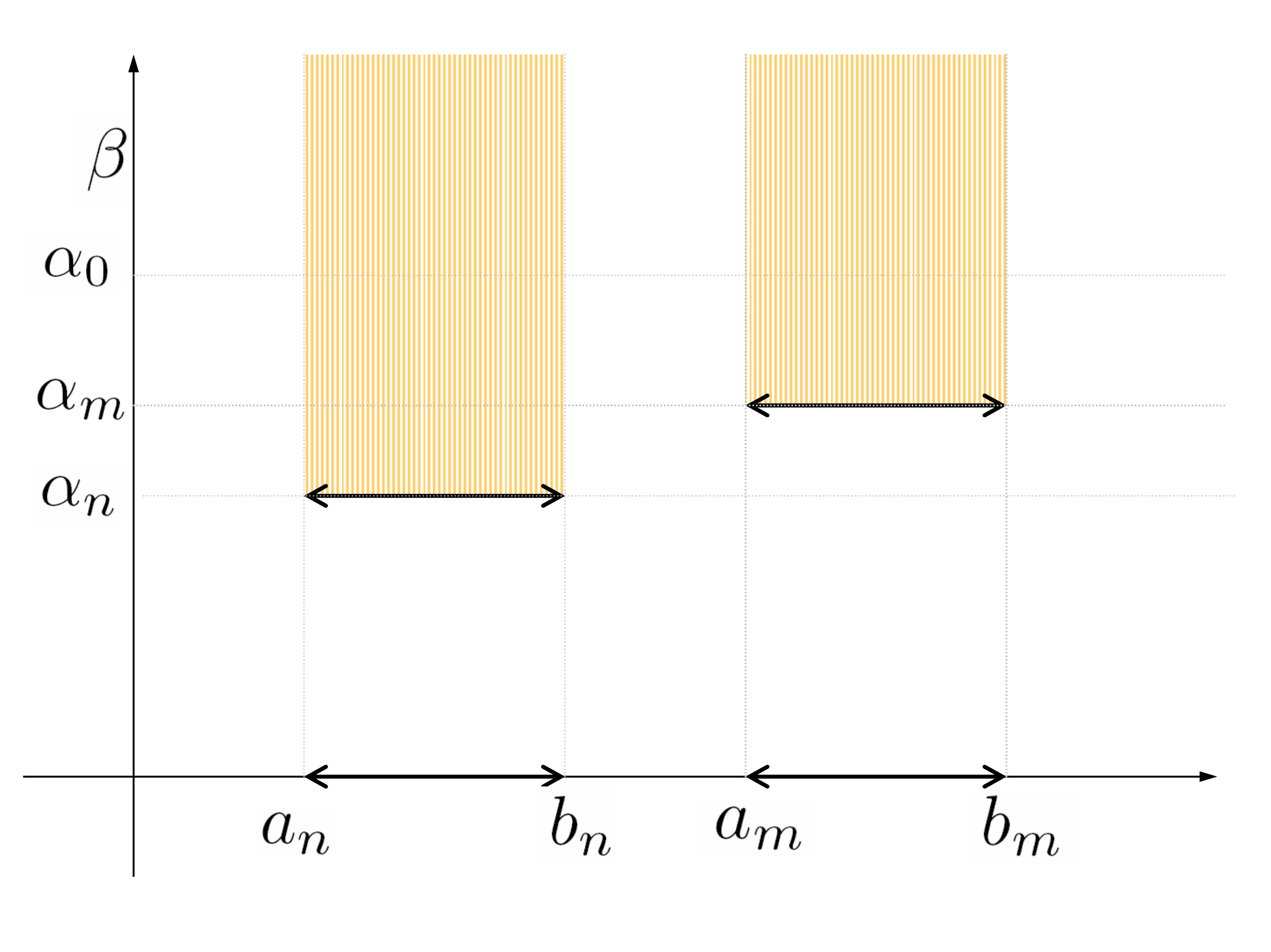}
\end{center}

Let us note that $\bigcup_{\al < \al_0} \Q_\al \times \{\al\}$ is countable since it is a countable union of countable sets.
Therefore we only need to show that $\bigcup_{{\al_0}\leq\al}\Q_\al \times \{\al\}$ can be covered with a countable number of elements of $\gamma$.
Let $(x,\be) \in X$. We have two cases:

\begin{case}[1]
$(x,\be) \in \Q_\al \times \{\al\}$ for some $\al < \al_0$. 
This means in particular that $\be=\al$. 
Then, since $\g$ covers $X$, we choose $O(a_\al, b_\al, \be_\al)$ in $\g$ such that $(x,\be) \in O(a_\al, b_\al, \be_\al)$.
\end{case}
\begin{case}[2]
Let $\be \geq \al_0$ and $(x,\be) \in X$, i.e. $x \in \Q_\be$.
Since $\g$ covers $X$, there is a $k \in K$ such that 
$(x,\be) \in O(a_k,b_k,\al_k) = \{(y,\de) : y \in \Q_\de,\ a_k < y<b_k , \al_k \leq \de < \w_1\}$.
Then $x \in \Q_\be \subset A$ and $\Upsilon$ covers $A$. 
Hence, there is an $n \in \N$ such that $x \in (a_n,b_n)$.
Then (by definition of the $O$'s), since $\be \geq \al_0 \rightarrow \be \geq \al_n$  for every $n\in\N$ and
$(x,\be) \in O (a_n,b_n,\al_n)= \{(y,\delta) : y \in \Q_\delta,  a_n < y<b_n , \al_n \leq \delta < \w_1\}$.
\end{case}
Hence, $X$ is Lindel\"of. 
\end{proof}

\begin{cl}X is CCC.
\end{cl}

\begin{proof}
We want to show that every family of pairwise disjoint nonempty open sets in $X$ is at most countable.
It is sufficient to show that any family of pairwise disjoint nonempty \emph{basic} open sets in $X$ is at most countable.
Let
\begin{displaymath}
\g = \{O(a_k,b_k,\al_k) : k \in K\}
\end{displaymath}
be any family of pairwise disjoint nonempty basic open sets in $X$. 
Suppose $K$ is uncountable.
Consider the family of open in $\R$ sets $\U = \{(a_k,b_k) : k \in K\}$. 
Since $\R$ is CCC, $\U$ cannot consist of pairwise disjoint sets.
So there are $k_1,k_2\in K$ such that $(a_{k_1}, b_{k_1}) \cap (a_{k_2}, b_{k_2})\neq\emptyset$. 
Let $\al = \max \{\al_{k_1},\al_{k_2}\}$.
Since $\Q_\al$ is dense in $\R$ we can find $y \in \Q_\al \cap (a,b)$.
Then $(y,\al) \in O(a_{k_1}, b_{k_1}, \al_{k_1} ) \cap O(a_{k_2}, b_{k_2}, \al_{k_2} )$ -- a contradiction. 
\end{proof}

\begin{cl}X is not separable.
\end{cl}

\begin{proof}

If $M = \{(x_n , \al_n) : n \in \N\}$ is an arbitrary countable subset of $X$, let us take

\begin{displaymath}
\al = \sup \{\al_n : n \in \N\} +1.
\end{displaymath}
Then $\sup \{ \al_n : n \in \N\}$ is countable as a countable union of countable ordinals, and hence $\al$ is, too. 
Then if $a,b \in \R$, $O(a,b,\al)$ is a nonempty open set in $X$ which does not intersect $M$.

\input{Winkler5pic.tex}
\\

Hence $X$ is not separable.
\end{proof}

\section{Lindel\"of-type Covering Properties} \label{Lindelofgeneralizations}

Now we shall define the \Lind-type covering properties we mentioned in the very beginning. 
We shall overview some statements for these spaces similar to the results for \Lin spaces.
\subsection{Almost Lindel\"of Spaces}

Almost \Lin spaces were considered in \cite{DW}, where Dissanayeke and Willard generalised  Archangelskii's result about the cardinality of Hausdorff first countable \Lin spaces.

\begin{defn} A topological space $X$ is \emph{almost Lindel\"of} if for every open cover $\U$ of $X$, there exists a countable subfamily $\V \subset \U$ with  $\bigcup_{V \in \V} \bar{V}=X$.
\end{defn} 
It follows immediately from the definition that  every Lindel\"of space is almost Lindel\"of. 
 Dissanayeke and Willard also stated (without proof) that:

\begin{propn} \label{regalisl} A regular almost Lindel\"of space is Lindel\"of.
\end{propn}

\begin{proof}
Let 
\begin{displaymath}
\U=\{U_\al:\al\in A\} 
\end{displaymath}
be an open cover of $X$. 
We choose a countable subcover in the following way:
for any $x\in X$ choose $\al(x)\in A$ with $x\in U_{\al(x)}$. 
By regularity, there is an open $V_x$ with $x\in V_x\subset\bar{V_x}\subset U_{\al(x)}$.
Then 
\begin{displaymath}
\V=\{V_x:x\in X\} 
\end{displaymath}
is an open cover of $X$. 
Since $X$ is almost Lindel\"of we can choose $\{x_n:n\in\N\}\subset X$ such that $\displaystyle{X=\bigcup_{n\in\N}\bar{V_{x_n}}}$.
But $\displaystyle{\bigcup_{n\in\N}\bar{V_{x_n}}\subset \bigcup_{n\in\N} U_{\al(x_n)}}$.
Hence $\{U_{\al(x_n)}:n\in\N\}$ is the required countable subcover of $\U$. 
\end{proof}

However, as should be expected, not every almost \Lin space is \Lind, and also, if we weaken the requirement about regularity in the above proposition, it fails to remain true.

We use ideas from Mysior (\cite{M}) and improve an example from \cite{SZ} in order to obtain the following:

\begin{xmpl} \label{0ex1}  There exists a Urysohn 
almost Lindel\"of space $X$ which is not Lindel\"of.
\end{xmpl}

\begin{construction}
Let $A=\{(a_\al, -1):\al<\w_1\}$ be an $\w_1$-long sequence in the set $\{(x,-1):x\geq 0\}\subseteq\R^2$.
Let $Y=\{(a_\al, n):\al<\w_1,n\in\w\}$. 
Let $a=(-1,-1)$.
Finally let $X=Y\cup A\cup\{a\}$.

We topologize $X$ as follows:
\begin{itemize}
\item[-]
all points in $Y$ are isolated;
\item[-]
For $\al<\w_1$ the basic neighborhoods of $(a_\al,-1)$ will be of the form 
\begin{displaymath}
U_n(a_\al,-1)=\{(a_\al,-1)\}\cup\{(a_\al, m):m\geq n\} \textrm{ for } n\in\w;
\end{displaymath}
\item[-]
The basic neighborhoods of $a=(-1,-1)$ are of the form
\begin{displaymath}
U_\al (a)=\{a\}\cup\{(a_\be , n):\be>\al, n\in\w\} \textrm{ for } \al<\w_1.
\end{displaymath}
\end{itemize}
Let us point out that $A$ is closed and discrete in this topology. 
Indeed,  
for any point $x\in X$ there is a basic neighborhood $U(x)$ such that $A\cap U(x)$ contains at most one point and also that $X\setminus A=\{a\}\cup Y$ is open (because $U_\al(a)\subset Y\cup\{a\}$.
Hence $X$ contains an uncontable closed discrete subset and therefore (by Proposition \ref{closed discrete subsets of HL}) it cannot be \Lind. 

Note that for any open $U\ni a$ the  set $X\setminus\bar{U}$ is at most countable. 
Indeed, for any $\al<\w_1$, $\bar{U_\al(a)}=U_\al(a)\cup\{(a_\be, -1):\be>\al\}$. 
Hence $X\setminus \bar{U_\al(a)}$ is at most countable.

Let us now prove that $X$ is almost \Lind.
Let $\U$ be an open cover of $X$. 
Then there exists a $U(a)\in \U$ such that $a\in U(a)$.
We can find a basic neighborhood $U_\be(a)\subset U(a)$.
Then  $\bar{U_\be(a)}\subset \bar{U(a)}$ and hence $X\setminus \bar{U(a)}$ will also be at most  countable.
Hence $X\setminus \bar{U(a)}$ can be covered by (at most) countably many elements of $\U$.
Therefore $X$ is almost \Lind.

It is easily seen that $X$ is Hausdorff.

$X$ is Urysohn because 
\begin{displaymath}
\bar{U_\al(a)}=U_\al(a)\cup\{(a_\be,-1):\be>\al\},
\end{displaymath}
\begin{displaymath}
\bar{U_n(a_\al,-1)}=U_n(a_\al,-1)\cup\{a\} \textrm{, and }
\end{displaymath}
\begin{displaymath}
\bar{U(a_\al,n)}=U(a_\al,n)=\{(a_\al, n)\}.
\end{displaymath}
\end{construction}

\subsubsection{The Almost Lindel\"of Property  and the Main Topological Operations}

As in the case of \Lin spaces, we have the following results (stated in \cite{SZ} without proofs):

\begin{propn} \label{ctsaL} A continuous image of an almost Lindel\"of space is almost Lindel\"of.
\end{propn}  
\begin{proof}
Let $f:X\rightarrow Y$ be continuous, $X$ be almost Lindel\"of and $\U$ be an open cover for $Y$. 
Then $\V=\{f^{-1}(U):U\in\U\}$ is an open cover for $X$, and hence it has a countable subset
$\V'\subset\V$ with $\displaystyle{\bigcup_{V\in\V'}\bar{V}}=X$. 
Then $\U'=\{U\in\U:f^{-1}(U)\in\V'\}$ is the required subcollection. 
Indeed, let $y\in Y$, then $f^{-1}(y)\subset X$, 
so there exists a subcollection $\V''\subset\V'$ with $\displaystyle{\bigcup_{V_n\in\V''}V_n}\supseteq f^{-1}(y)$.
Then $f(\bar{V_n})=f(\bar{f^{-1}(U_n)})\subseteq \bar{U_n}$.
So $\displaystyle{y\in \bigcup_{f^{-1}(U_n)\in\V''}\bar{U_n}}$.
\end{proof}

\begin{propn} Quotient spaces of almost \Lin spaces are almost \Lin (since the quotient map is continuous).
\end{propn}

\begin{propn} The countable disjoint sum $\bigoplus_{n \in \N} X_n$ is almost Lindel\"of if and only if all spaces $X_n$ are almost-Lindel\"of.
\end{propn}

For the inheritance of the almost \Lin property, we have a weaker result than in the \Lin case:

\begin{propn} If $X$ is almost Lindel\"of, then any clopen subset of $X$ is almost Lindel\"of.
\end{propn} 
\begin{proof}
Let $F\subset X$ be clopen, and $\U$ be a cover for $F$.
Then $\U\cup(X\setminus F)$ is an open cover for $X$, hence there is a countable subfamily 
 $\U'$ of $\U$
 such that $\displaystyle{ X=\bigcup_{U\in\U'}\bar{U}}\cup\bar{X\setminus F}$.
Since $X\setminus F$ is also clopen,  $\U'$ is the required subcollection. 
\end{proof}

Let us point out that (as in the \Lin case), the product of even two almost \Lin spaces might not be almost \Lin, and even more:
 the product of two \Lin spaces might not be almost \Lind. We have the following example: 

\begin{xmpl}
\label{SxSnotaL}
$\Sorg\times\Sorg$ is not almost \Lind.
\end{xmpl}

\begin{proof}
We have seen that $\Sorg$ is \Lind.
$\Sorg\times\Sorg$ is not almost \Lin because it is regular (as a product of two regular spaces) and not \Lind.
But we showed in Proposition \ref{regalisl} that in regular spaces, \Lin and  almost \Lin coincide.
\end{proof}

However, in \cite{SZ}, it is shown that:

\begin{thm} \label{productWLCpct} If $X$ is almost Lindel\"of and $Y$ is compact, then $X \times Y$ is almost Lindel\"of.
\end{thm}  
\begin{proof}
It is sufficient to prove that for any cover $\U$ of basic open 
sets in $X\times Y$, i.e. $\U\subseteq\{U\times V:U\in\tau_X, V\in\tau_Y\}$ we can choose the required subfamily.

For an arbitrary but fixed $x\in X$, 
 $\{x\} \times Y$ is a compact subset of $X \times Y$.
Hence, there exists a finite subcollection $\{U_{x_i} \times V_{x_i} : i=1,2,... n_x\}$ of $\U$ such that
\begin{displaymath}
\{x\} \times Y \subseteq \bigcup \{U_{x_i} \times V_{x_i} : 1 \leq i \leq n_x\}. 
\end{displaymath}
Let $W_x = \bigcap \{U_{x_i} : 1 \leq i\leq n_x\}$ (hence $W_x$ is open in $X$ and $x\in W_x$). 
Then again
\begin{displaymath}
\{x\} \times Y \subseteq \bigcup \{W_{x} \times V_{x_i} : 1 \leq i \leq n_x\}. 
\end{displaymath}
 $\mathcal{W} = \{W_x : x \in X\}$ is an open cover of $X$. 
Since $X$ is almost Lindel\"of, we can choose a countable subfamily $\{W_{x_j} : j \in \N\}$ of $\mathcal{W}$ whose
closures cover $X$. 
We claim that
\begin{displaymath}
\V = \{U_{x_{j_i}}\times V_{x_{j_i}}: 1 \leq i \leq n_{x_j} , j \in \N\}.
\end{displaymath}
 is the required countable subcollection of $\U$. 

So let $(s,t)\in X\times Y$ be arbitrary but fixed. 
Then $s\in X=\bigcup_{j \in \N} \bar{W}_{x_j}$, hence there exists an $x_{j_0}$ such that $s\in \bar{W}_{x_{j_0}}$. 
Moreover, since $\{W_{x_{j_0}} \times V_{x_i} : 1\leq i \leq n_{x_{j_0}}\}$ covers $\{x_{j_0}\}\times Y$, there is an $i_0$
with $(x_{j_0},t)\in W_{x_{j_0}} \times V_{x_{i_0}}$.
And since $s\in \bar{W}_{x_{j_0}}$, we have that $(s,t) \in \bar{ W}_{x_{j_0}} \times V_{x_{i_0}}$.
By definition of $W_{x_{j_0}}$, we have that there is a $U_{x_{j_0}}$ with $W_{x_{j_o}}\subset U_{x_{j_0}}$, and 
hence $(s,t) \in \bar{U}_{x_{j_0}} \times V_{x_{i_0}}\subset \bar{U_{x_{j_0}} \times V_{x_{i_0}}}$.
\end{proof}

\subsection{Weakly Lindel\"of Spaces}

Weakly \Lin spaces were introduced in 1959 by Frolik in \cite{Fro}. 
 Here, we consider interrelations between the weakly \Lin property and other topological properties, as well as its preservation (or destruction) under subsets, products, and continuous maps.

\begin{defn} \cite{Fro} $X$ is  \emph{weakly Lindel\"of} if for every open cover $\U$ of $X$, there exists a countabe subset $\V$ of $\U$ with $\bar{\bigcup_{V \in \V} V} = X$. 
\end{defn}

\begin{xmpl} Any Lindel\"of (and, as we shall see, any CCC, and hence separable) space is weakly Lindel\"of. 
\end{xmpl}

\begin{xmpl} Any uncountable set $X$ with the discrete topology is not weakly Lindel\"of. 
\end{xmpl}

Bell, Ginsburg, and Woods in \cite{BGW} considered the weakly \Lin property in the settings of cardinal invariants of topological spaces. Here, we present the proof of the following theorem which they only stated:

\begin{thm} \label{CCCwL}If $X$ is CCC then $X$ is weakly Lindel\"of. 
\end{thm}

\begin{proof}  
Suppose that $X$ is not weakly Lindel\"of, i.e.
that there exists an uncountable open cover $\G$ of $X$ with nonempty sets such that 
for each of its countable subfamilies $\G' \subset \G$, we have that $X\setminus \bar{\bigcup \G'} \neq \emptyset$. 
We then construct a countable chain of disjoint open nonempty sets indexed by $\g,\g < \w_1$, hence showing that $X$ is not CCC. 

Let us write $\G=\{U_\al:\al<\be\}$ for some uncountable ordinal $\be$.
Pick $U_0 \in \G$. 
Then $X\setminus\bar{U_0}\neq \emptyset$. 
So, we have that $\displaystyle{X \setminus \bar{U_0} \subseteq \bigcup_{0<\al<\be} U_\al}$.  

Since $X\setminus\bar{U_0}\neq \emptyset$, there exists an $\al_1 > 1$ with $U_{\al_1} \cap( X\setminus\bar{U_0}) \neq \emptyset$. 
Let $V_1=U_{\al_1} \cap (X\setminus\bar{U_0})$. 
Then $V_1 \neq \emptyset$, $V_1$ is open, and $V_1 \cap V_0 = \emptyset$.

Now, $\displaystyle{X \setminus \bar{U_0 \cup U_{\al_1}} \subseteq \bigcup_{0<\al<\be,\al\neq \al_1} U_\al}$, and 

\begin{displaymath}
X \setminus \bar{U_0 \cup U_{\al_1}} = \bigcup_{0<\al<\be,\al\neq \al_1} U_\al \cap (X \setminus \bar{U_0 \cup U_{\al_1}})
\end{displaymath}

Hence, there exists a $U_{\al_2} \in \G$ such that $U_{\al_2} \cap \left( X \setminus \bar{U_0 \cup U_{\al_1}}\right) \neq \emptyset$. 
Then, we define $V_2=U_{\al_2} \in \G$ as $U_{\al_2} \cap \left( X \setminus \bar{U_0 \cup U_{\al_1}}\right)$. Obviously, $V_2$ is open and disjoint from $V_0,V_1$. 

We continue inductively as follows.
Let $\g < \w_1 (\leq \be)$. 
Suppose we have already constructed $V_\de \neq \emptyset$ open in $X$ such that the family $\{V_\de : \de < \g\}$ is disjoint.
Then $X\setminus \bar{\bigcup_{\de < \g} U_{\al_\de}} \neq \emptyset$ because $\g$ is a countable ordinal (and our assumption is that $X$ is not weakly Lindel\"of). 
Then, 
\begin{displaymath}
X\setminus \bar{\bigcup_{\de < \g} U_{\al_\de}} = \bigcup_{\al<\be, \al \notin \{\al_\de : \de<\g\}}U_\al\cap \left(X \setminus \bar{\cup_{\de<\g}U_{\al_\de}}\right).
\end{displaymath}
Hence, as the left side of the equality is nonempty, so is the right, so we can choose
$U_{\al_\g}$ such that $U_{\al_\g} \cap \left(X \setminus \bar{\cup_{\de<\g}U_{\al_\de}}\right)\neq \emptyset$.
Then, we define $V_\g = U_{\al_\g} \cap \left(X \setminus \bar{\cup_{\de<\g}U_{\al_\de}}\right) \neq \emptyset$. 
Hence, $V_\g$ is nonepmty, open, and by construction, we have that $V_\g \cap V_\de = \emptyset$ for every $\de<\g$.

In this way,  we have constructed the uncountable family - $\{V_\g : \g < \w_1\}$ - 
of disjoint nonempty open sets, hence showing that $X$ is not CCC. 
\end{proof}

\begin{propn} \label{aLwL} Every almost Lindel\"of space is weakly Lindel\"of. 
\end{propn}

\begin{proof} Let $X$ be almost Lindel\"of, and $\U$ be an open cover of $X$.
Then there exists a countable subset $\V$ of $\U$ with $ \bigcup_{V \in \V} \bar{V} = X$. 
As $\bigcup_{V \in \V} \bar{V} \subset \bar{\bigcup_{V \in \V} V}$, we also have that $X$ is weakly Lindel\"of. 
\end{proof}

As the following example shows, the converse is not true.

\begin{xmpl}\label{inverse}
The Sorgenfrey plane is regular weakly Lindel\"of, but not almost Lindel\"of (and hence, not \Lin).
\end{xmpl}

\begin{proof}
Example \ref{SxSnotaL} shows that  $\Sorg \times \Sorg$ is not almost Lindel\"of.
However, it is separable, since $\Q\times\Q$ is dense in $\Sorg \times \Sorg$, and hence $\Sorg \times \Sorg$ is CCC. 
By Theorem \ref{CCCwL}, it is weakly Lindel\"of. 
\end{proof}

This example also shows, that, unlike in the almost \Lin case, in regular spaces the weakly \Lin and \Lin properties do not coincide. 

Next, we present another example of a weakly \Lin not \Lin space, which is given in \cite{BGW}. It is not regular, but, unlike $\Sorg\times\Sorg$, it has arbitrary cardinality. We will show that it has some additional properties, besides the ones mentioned in the article. 
We show that in this space, the weakly \Lin property is not inherited by  closed subspaces, which is a point of difference from \Lin spaces. 
We point out that it is not CCC (hence not separable), so the converse of Theorem \ref{CCCwL}
 does not hold. 
We also show that it does not possess another \Lind-type covering property, i.e. it is not quasi-\Lind.
The definition of the latter is the following:

\begin{defn} \label{qL} A space is called \emph{quasi-Lindel\"of} if every closed subset of it is weakly Lindel\"of. 
\end{defn} 
This notion was introduced by Archangelski in \cite{A}. We will consider it in more detail in section \ref{quasi-Lindelof}.

\begin{xmpl} \cite{BGW} \label{BGW} A Hausdorff, first countable, weakly Lindel\"of space that is neither Lindel\"of nor quasi-Lindel\"of. 
\end{xmpl}

\begin{construction}
We will denote the irrational numbers by $\I$. 
Let $\kappa$ be an arbitrary uncountable cardinal number, $A\subset\I$ be countable and dense in $\I$ (and hence also dense in $\R$). 
Let $Z=(\Q\times\ka)\cup A$.
Note $A\cap \Q=\emptyset$.
For every $(q,\al)\in\Q\times\ka$ define a neighborhood base
\begin{displaymath}
U_n(q,\al)=\{(r,\al):r\in\Q\wedge|r-q|<\n   \},\ n\in\N.
\end{displaymath}
For every $a\in A$, define a neighborhood base
\begin{displaymath}
U_n(a)=\{b\in A:|b-a|<\n\}\cup\{(q,\al):\al<\ka\wedge|q-a|<\n\}, \ n\in\N.
\end{displaymath}

\includegraphics[width=15cm]{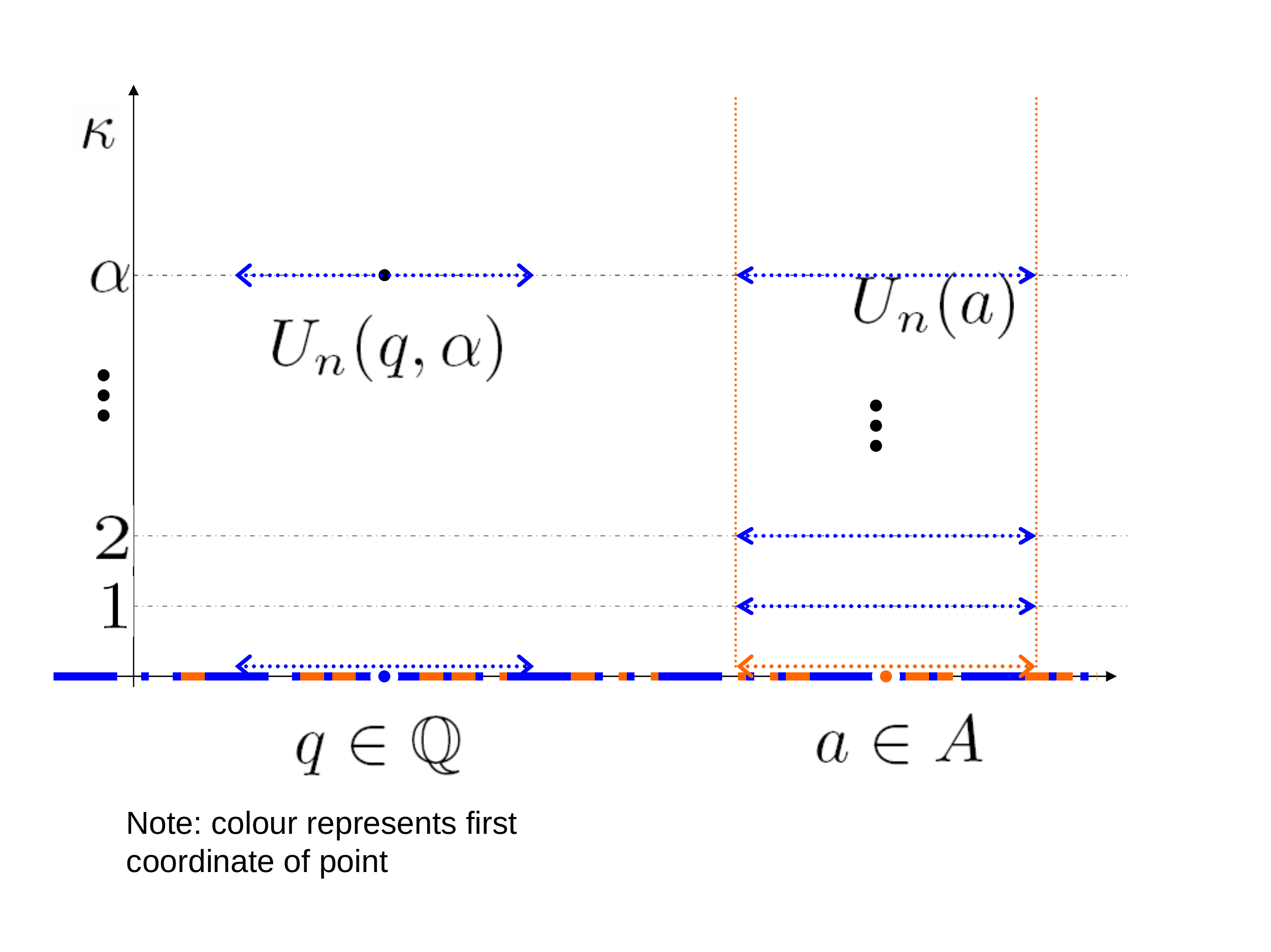}

Let $\tau$ be the topology generated by $\{U_n(q,\al):(q,\al)\in\Q\times\ka\}\cup \{ U_n(a):a\in A, n\in\N\}$.
Then $(Z,\tau)$ is first countable by construction.
\end{construction}
\begin{cl}
$Z$ is Hausdorff.
\end{cl}
\begin{proof}
\begin{case}[1]
$a_1\neq a_2$, $a_1,a_2\in A$. 
Then they can be separated in $\R$ (which is Hausdorff) by disjoint open intervals $(a_1-\n, a_1+\n)\cap (a_2-\n,a_2+\n)$.
Then $U_n(a_1)\cap U_n(a_2)=\emptyset$.
\end{case}
\begin{case}[2]
$(q_1,\al_1)\neq(q_2,\al_2)$. 

If $q_1\neq q_2$, there exists $n\in\N$ such that $(q_1-\n,q_1+\n)\cap(q_2-\n,q_2+\n)=\emptyset$ in $\R$.
Hence $U_n(q_1,\al_1)\cap U_n(q_2,\al_2)=\emptyset$. 

If $q_1=q_2$ then, without loss of generality, $\al_1\lneqq \al_2$. 
Then for every $n$, $U_n(q_1,\al_1)\cap U_n(q_2,\al_2)=\emptyset$.
\end{case}
\begin{case}[3]
$q\neq a$.
Then again there is an $n$ with $(q-\n,q+\n)\cap(a-\n,a+\n)=\emptyset$ and
$U_n(a)\cap U_n(q,\al)=\emptyset$ for every $\al<\ka$.
\end{case}
\end{proof}
\begin{nt}
For each $\al<\ka$ the subspace $\Q\times\{\al\}$ is open in $Z$ (because 
every point in $\Q\times\{\al\}$ has a neighborhood that consists only of elements of $\Q\times\{\al\}$,
and also $\Q\times\{\al\}$ is homeomorphic to $\Q$, as a subspace of $\R$).
\end{nt}
\begin{cl}
$Z$ is weakly Lindel\"of.
\end{cl}
\begin{proof}
We will first show that if $G\subset Z$ is open and $A\subset G$, then $G$ is dense in $Z$.

So, we have to show that $G$ intersects any  nonempty basic open set in $Z$.
Since $A\subset G$, $G$ will intersect any neighborhood of points in $A$, so we only have to prove that
if $q\in\Q$, $\al<\ka$ and $n\in\N$ then $U_n(q,\al)\cap G\neq\emptyset$.

Recall that $A$ is dense in $R$. 
Hence for every nonempty interval $(q-\n,q+\n)\subset \R$, there exists $a\in A$ such that $a\in(q-\n,q+\n)$, so
$|q-a|<\n$.

Since $G$ is open and $a\in A\subset G$, there is $m\in\N$ with $U_m(a)\subseteq G$.

Let $r\in\Q$ be such that $|a-r|<\min \{\frac{1}{m},\frac{1}{n}-|q-a|\}$. 
Then $(r,\al)\in G$ and 
\begin{displaymath}
|q-r|\leq |q-a|+|a-r|<|q-a|+\frac{1}{n}-|q-a|=\frac{1}{n},
\end{displaymath}
so $(r,\al)\in U_n (q,\al)$.

Hence $(r,\al)\in G$ and $(r,\al)\in U_n(q,\al)$.
Hence $(r,\al)\in G\cap U_n(q,\al)$, and so $G$ is dense in $Z$. 

Then $Z$ is weakly Lindel\"of. 
Indeed let $\U$ be an arbitrary open cover of $Z$. 
For every $a\in A$ let $U_a \in\U$ such that $a\in U_a$.
Let $\V=\{U_a:U_a\in\U,a\in A\}$, then $\V$ is countable and $G=\bigcup\V$ is open; hence 
$\bar{\bigcup\V}=\bar{G}=Z$.
\end{proof}

\begin{cl} $Z$ is not Lindel\"of.
\end{cl}
\begin{proof}
For every $q\in\Q$ the set $\{q\}\times \ka$ is uncountable, closed and discrete in $Z$:
$\{q\}\times\ka$ is discrete because for any $n\in\N$ and $\al<\ka$ we have $U_n(q,\al)\cap(\{q\}\times \ka)=(q,\al)$.

\includegraphics[width=15cm]{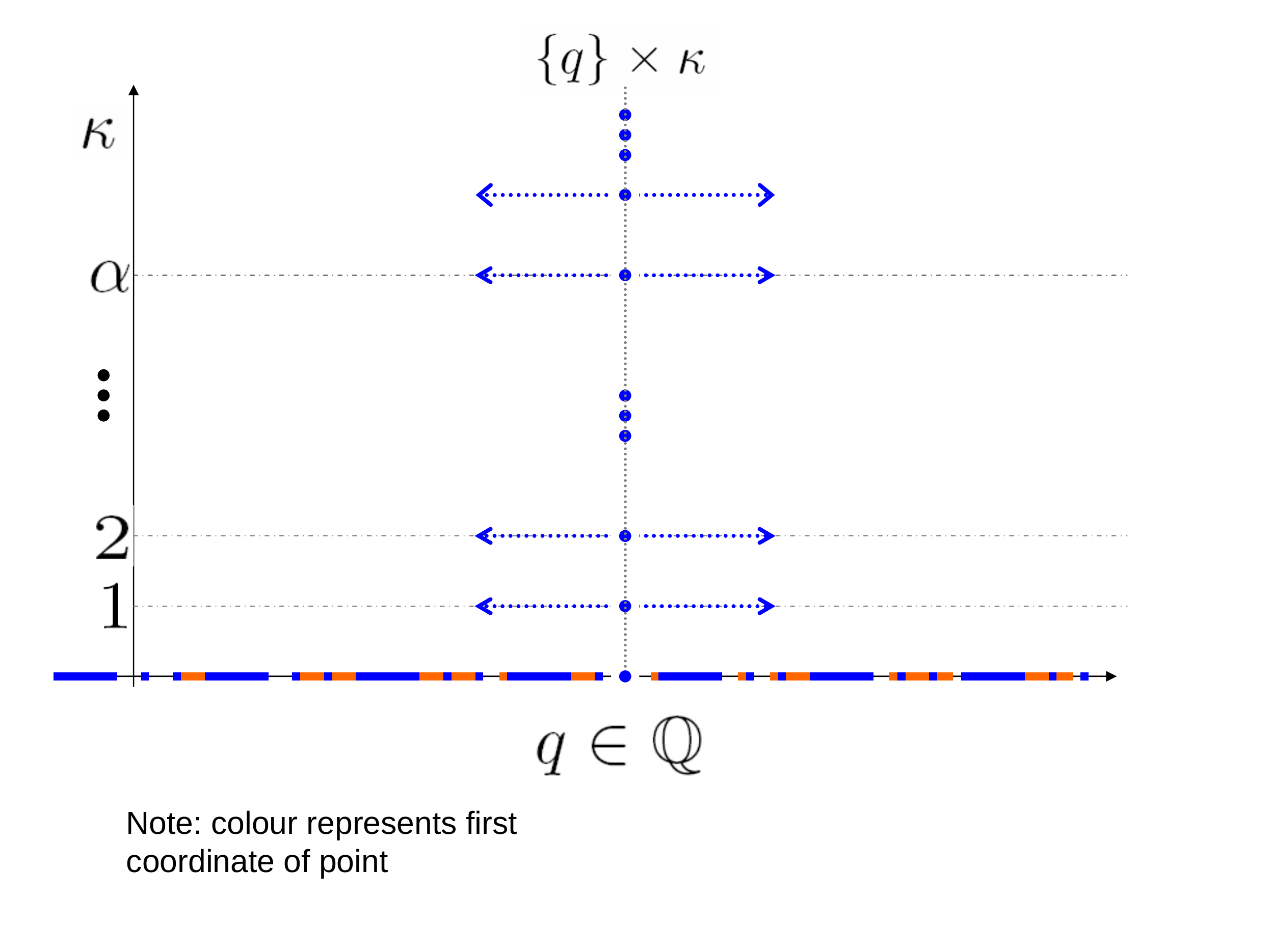}

$\{q\}\times \ka$ is closed because for each $x\in Z\setminus(\{q\}\times \ka)$ there is a neighborhood which does not intersect $\{q\}\times \ka$. 
Therefore, by Lemma \ref{closed discrete subsets of HL}
we have that $Z$ is not Lindel\"of. 
\end{proof}

\begin{cl}
$Z$ is not CCC.
\end{cl}
\begin{proof}
Let us fix $q\in\Q$.
Then the family $\U=\{U_n(q,\al):\al<\kappa\}$ is an uncountable family of disjoint  open sets. 
\end{proof}

\begin{cl}$Z$ is not quasi-Lindel\"of.
\end{cl}
\begin{proof}
Again, $\{q\}\times\ka$ is closed and discrete and from the cover $\g=\{(x-\n,x+\n)\times\al:\al<\ka\}$, 
no countable subfamily can be chosen, the closure of whose union covers  $\{q\}\times\ka$.
This is because 
$\bar{\bigcup\g}$ consists of 
$\{[q-\n,q+\n]\times\{\al\}:\al<\ka\}\cup(A\cap([q-\n,q+\n]\times\{0\}))$
and $([q-\n,q+\n]\times\{\al\})\cap([q-\n,q+\n]\times\{\be\})=\emptyset$ if $\al\neq\be$. 
Hence if we remove $[q-\n,q+\n]\times\{\al\}$  for some $\{\al\}$, the point $(q,\al)$ will remain uncovered. 
\end{proof}

\subsubsection{The Weakly Lindel\"of Property and the Main Topological Operations}

Let us prove the following four results stated in \cite{SZ}.

\begin{defn}
A subset $A$ of a topological space $X$ is called \emph{regularly closed} if $A=\bar{\Int(A)}$
\end{defn}

\begin{propn}\label{regclwLwL}Every regularly closed subset of a weakly Lindel\"of space $X$ is weakly Lindel\"of.
\end{propn}
\begin{proof}
Let $F\subset X$ be regularly closed, and $\U$ be an open cover for $F$.
Then $\U\cup(X\setminus F)$ is an open cover for $X$. 
Hence there exists $\U'\subset \U$ with 
\begin{displaymath}
\bar{\U'\cup(X\setminus F)}=X.
\end{displaymath}
Now, $\Int(F)\cap\bar{(X\setminus F)}=\emptyset$, since $F$ is regular closed. Hence, $\Int(F)\subseteq\bar{\bigcup\U'}$.
Therefore $F=\bar{\Int(F)}\subseteq \bar{\bigcup\U'}$.
\end{proof}

\begin{propn} Every clopen subset of a weakly \Lin space is weakly \Lind.
\end{propn}

\begin{proof}
This follows from Proposition \ref{regclwLwL}, since clopen subsets are regularly closed sets. 
\end{proof}

As we have already pointed out in Example \ref{BGW}, not every closed subset of a weakly \Lin space is weakly \Lind, so the two results given above are the `best possible' with respect to the inheritance of the weakly \Lin property. 

\begin{propn} \label{ctswL} The continuous image of a weakly Lindel\"of space is weakly Lindel\"of. 
\end{propn} 

\begin{proof}
Let $f:X\rightarrow Y$ be continuous, $X$ be weakly Lindel\"of, and $\U$ be an open cover of $Y$. 
Then $\V=\{f^{-1}(U):U\in\U\}$ is an open cover for $X$. 
Hence there exists a countable subfamily $\V'\subset\V$ with $\bar{\bigcup\V'}=X$. 
Then $\U'=\{U:f^{-1}(U)\in\V'\}\subset\U$ is the required subcollection.

Indeed, let $W\subset Y$ be open. 
Then $f^{-1}(W)$ is open, hence there is $V\in\V'$ with $f^{-1}(W)\cap V\neq\emptyset$, say $x\in f^{-1}(W)\cap V$.
Then $f(x)\in f(f^{-1}(W)\cap V)=W\cap V$, so $\bigcup\U'=Y$, as required. 
\end{proof}

\begin{cor} A quotient space of a weakly Lindel\"of spaces is weakly Lindel\"of.
\end{cor}

\begin{propn} The countable disjoint sum $\oplus_{n \in \N} X_n$ is weakly Lindel\"of iff all spaces $X_n$ are weakly-Lindel\"of.
\end{propn} 
\begin{proof}
In disjoint sums, $\bar{\bigcup_{U\in\U}U}=\bigcup_{n\in\N} \bar{\U\times X_n}$.
\end{proof}

Now, we shall present an example of Hajnal and Juhasz  showing that the product of even two weakly \Lin spaces need not be weakly \Lind.

\begin{xmpl} \label{HJ} \cite{HJ}A Hausdorff not weakly Lindel\"of space that is a product of two Lindel\"of spaces. 
\end{xmpl} 

\begin{rmk}
 The example is based on some properties of the topology on linearly ordered spaces. So, here we give the necessary background.
\end{rmk}

\begin{defn} The order $``<"$ is said to be \emph{dense} iff whenever $x\lneqq y$ there is $z\in X$ such that $x\lneqq z \lneqq y$. 
\end{defn}

\begin{defn} $``<"$ is said to be \emph{complete} if every non-empty subset of $X$ that has an upper bound has a least upper bound. We shall denote it by $\sup X$. 
\end{defn}

\begin{defn} If $\left<X,<\right>$ is a linearly ordered set, we can define a topology \emph{$\tau(<)$} with the base consisting of all open (with respect to $``<"$) intervals in $X$. 
\end{defn}
Let $(X,<)$ be any linearly ordered set. Let $X^{+}$ be $X$ with the topology generated by 
the basis consisting of half-open intervals of the form $[x,y)$ and 
 $\Xm$ be $X$ with the topology generated by the basis consisting of half-open intervals of the form $(x,y]$.

The underlying set for Example \ref{HJ} is the well-known Lexicographic Square (example 48, page 73 of \cite{SS}).
Here, we provide details and proofs for properties which have been only stated in \cite{SS}.

\begin{xmpl}[The Lexicographic Square] Let $X=[0,1]\times[0,1]\subset\R^2$. 
Define an order $``\prec"$ on $X$ in the following way:

$(a,b)\prec(c,d)$ iff $a<c$ or ($a=c$ and $b<d$), where $``<"$ is the usual order topology on $X$.

Let $\tau(\prec)$ be the corresponding order topology on $X$. 
\end{xmpl}

\begin{cl} $\left(X,\tau(\prec)\right)$ is first countable.
\end{cl}

\begin{proof}
Let us note that if $(x,y)\in X\setminus \left(\left( [0,1]\times\{0\} \right) \cup \left( [0,1]\times\{1\}\right)\right)$, then
$O_{\frac{1}{n}}(x,y)=\{x\}\times(y-\frac{1}{n},y+\frac{1}{n})$, for $n\in\N$ form a countable neighborhood base
for $\tau(<)$ at $(x,y)$. 
They are open intervals since 
\begin{displaymath}
O_{\frac{1}{n}}(x,y)=((x,y-\frac{1}{n}),(x,y+\frac{1}{n}))=
\{(u,v)\in X: (x,y-\frac{1}{n}) \prec (u,v)\prec (x,y+\frac{1}{n})\}.
\end{displaymath}
Indeed, if $(u,v)$ is such that $(x,y-\frac{1}{n}) \prec (u,v)\prec (x,y+\frac{1}{n})$ 
 we should have $x\leq u\leq x$, i.e. $u=x$, and $y-\frac{1}{n} < v < y+\frac{1}{n}$, as required. 
Also if $(x,y)\in U$ is open in $\tau(\prec)$, then there exists an $n$ 
such that $O_{\frac{1}{n}} (x,y) \subset U$.

The neighborhoods of $(x,1)$ are
\begin{displaymath}
O_{\frac{1}{n}}(x,1)=(
(x,x+\frac{1}{n})\times [0,1])\cup
(\{x\} \times (1-\frac{1}{n},1])\cup
(\{x+\frac{1}{n}\} \times [0,\frac{1}{n})). 
\end{displaymath}

The neighborhoods of $(x,0)$ are
\begin{displaymath}
O_{\frac{1}{n}}=(
(x,x-\frac{1}{n})\times [0,1])\cup
(\{x-\frac{1}{n}\} \times (1-\frac{1}{n},1])\cup
(\{x\} \times [0,\frac{1}{n})).
\end{displaymath}
 Since $\{O_{\frac{1}{n}}(x,1):n\in\N\}$ and $\{O_{\frac{1}{n}}(x,0):n\in\N\}$ are also countable local bases at $(x,1)$ and $(x,0)$ respectively, it follows that $X$ is first countable.
\end{proof}

Now we proceed with the exposition of Example \ref{HJ}. 
Firstly, we  provide a proof for the following two statements of Hajnal and Juhasz:

\begin{cl} ``$\prec"$ is a complete order. 
\end{cl}

\begin{proof} 
Note that any subset $A\subset X$ is trivially bounded by $(1,1)$. 
Hence, we want to show that any subset $A\subset X$ has a least upper bound.

Let 
\begin{displaymath}
A_x= \{x\in [0,1]:\exists y\in [0,1] , (x,y)\in A\}.
\end{displaymath}
Since $A_x\subset [0,1]$, it is a bounded set of reals and hence $\al=\sup A_x$ exists. 

We have two cases:
\begin{case}[1]
 There exists  $y\in[0,1]$ such that $(\al,y)\in A$. 
Then the set $\{y\in[0,1]:(\al,y)\in A\}$ is nonempty and bounded (as a subset of $[0,1]$).
So let $\be=\sup\{y:(\al,y)\in A\}$. 
Then $(\al,\be)=\sup A$.
Indeed, let $(a,b)\in A$ Then $a\in A_x$, and so $a\leq \al$. 
If $a<\al$ then $(a,b)\prec (\al,\be)$.
If $a=\al$ then $b\leq \be$ as $(a,b)=(\al,b)\in \{y\in[0,1]:(\al,y)\in A\}$.
Hence $(a,b) \prec (\al,\be)$. 

Also, $(\al,\be)$ is the least upper bound: for suppose that $(a,b)\leq(\al,\be)$ is an upper bound. Then $a=\al$, since otherwise $(\al,y)>(a,b)$ for the supposed $y$ above. 
Also, if $b<\be$, then (since $\be$ is defined as a supremum), 
there is a $b'\in \{y\in[0,1]:(\al,y)\in A\}$ with $b<b'<\be$, so $(\al,b)<(\al,b')\in A$ and hence $(\al,b)$ is not an upper bound for $A$ after all. So $b=\be$.
\end{case}
\begin{case}[2]
 No point in $A$ has first coordinate $\al$, so
$(\{\al\}\times[0,1])\cap A = \emptyset$.
Then $\sup A=(\al,0)$.
\end{case}
First let us show that $(\al,0)$ is an upper bound for $A$.
Let $(a,b)\in A$. 
Then, by choice of $\al$, we have that $a\leq \al$. 
By assumption, $a\neq \al$, so we have that $a<\al$ and $(a,b)\prec(\al,0)$, as required. 

Let us show that $\sup A=(\al,0)$. 
It suffices to show that for every $n\in\N$, $(\al-\frac{1}{n},0)$ is not an upper bound for $A$. 
Let $n\in\N$; then $\al-\frac{1}{n}$ is not an upper bound for $A_x$. 
Hence there exists $a_n\in A_x$ and there exists $y_n\in [0,1]$ such that
$(a_n,y_n)\in A$ and $\al-\frac{1}{n}\leq a_n<\al$.
Then $(\al-\frac{1}{n},0)\preceq(a_n,y_n)$ for every $n\in\N$. 
Indeed: if $\al-\frac{1}{n} < a_n$ then $(\al-\frac{1}{n},0)\prec(a_n,y_n)$;
if $\al-\frac{1}{n}=a_n$ then $(\al-\frac{1}{n},0)=(a_n,0)\preceq (a_n,y_n)$.
Hence, for every $n\in\N$ we have found $(a_n,y_n)\in A$ such that $(\al-\frac{1}{n},0)\preceq (a_n,y_n)$.
Hence, $(\al-\frac{1}{n},0)$ is not an upper bound for $A$. 

Hence, $(\al,0)=\sup A$.
\end{proof}

\begin{cl}
$``\prec"$ is a dense order.
\end{cl}

\begin{proof}
Suppose $(a,b)\precneqq (c,d)$. We have two cases:

\begin{case}[1]
$a<c$. Then there exists  $u\in\R$ such that $a<u<c$. 
Then for any $v\in[0,1]$, we have that
$(a,b)\precneqq (u,v)\precneqq(c,d)$.
\end{case}

\begin{case}[2] $a=c$ and $b<d$. 
Then there exists  $v\in\R$ with $b<v<d$.
Then $(a,b)\precneqq(a,v)\precneqq(a,d)=(c,d)$.
\end{case}
\end{proof}

\begin{cl} $(X,\tau(\prec))$ is not CCC (and hence not separable). In fact, any dense subset of $X$ has cardinality $2^{\aleph_0}$. 
\end{cl}

\begin{proof}
Let $D$ be any dense subset of $X$. 
Since $|D|\leq|X|=2^{\aleph_0}$, it suffices to show that if $D$ is dense then $|D|\geq 2^{\aleph_0}$.

To do this, it suffices to construct a family of $2^{\aleph_0}$-many disjoint non-empty open substes of $X$ (hence showing that $X$ is not CCC).
Let 
\begin{displaymath}
\cC = \{\{x\}\times(0,1):x\in[0,1]\}.
\end{displaymath}
Then $|\cC|=|[0,1]|=2^{\aleph_0}$ and all elements of $\cC$ are open in $(X,\tau(\prec))$. 
Since every element of $\cC$ must contain a point of $D$, we have that $|D|\geq2^{\aleph_0}$.
\end{proof}

\begin{cl}
$X$ does not contain an uncountable set $A=\{a_\al=(x_\al,y_\al):\al<\w_1\}$ of length $\w_1$ which is increasing with respect to $``\prec"$.
\end{cl}

\begin{proof}
Suppose the contrary: that $A$ is such a set, and let $a=\sup A$. 
Then the space $A\cup\{a\}$, as a subspace of $X$, is homeomorphic to $[0,\w_1]$ with the order topology. 
But $A\cup \{a\}$ is first countable as a subspace of the first countable space $X$. 
On the other hand, $[0,\w_1]$ is not first countable. 
 For if $\{(w_n,\w_1]:n\in\N\}$ is a contable base at $\w_1$, then 
we would have that $\sup\{w_n:n\in\N\}=\w_1$, which is a contradiction, because $\w_1$ does not have a countable cofinal subset. 
Thus $X$ cannot contain an increasing sequence of length $\w_1$.
\end{proof}

\begin{cl}
In $X$ there is no decreasing uncountable subset of size $\w_1$.
\end{cl}

\begin{proof}
In a similar way as above, we can prove that any $A\subset X$ has a greatest lower bound, so $\inf A$ exists.
Also, we can analogously prove that if $A$ is a decreasing uncountable subset of size $\w_1$ in $X$ and $a=\inf A$, then
$\{a\}\cup A$ would again be homeomorphic to $[0,\w_1]$ (but this time with the topology on $\{a\}\cup A$ generated by the reverse of $``\prec"$ order) -- a contradiction.
\end{proof}

\begin{lem} 
Let $(X,<)$ be an order-complete linearly ordered set in which there is no decreasing or increasing ordered subset of type $\w_1$. 
Then both $\Xp$ and $\Xm$ are Lindel\"of topological spaces. 
\end{lem}

\begin{rmk}
Hajnal and Juhasz presented a more general case. 
Here we consider in detail the partial case for \Lin spaces.
\end{rmk}

\begin{proof}
Since $\Xp$ and $\Xm$ are homeomorphic (the reverse order in $\Xp$ gives $\Xm$ and vice-versa),
it suffices to show that $\Xp$ is Lindel\"of. 

Let $\U$ be an open cover of $\Xp$ by basic open sets of the form $[x,y)$. 

First let us show that:
\begin{cl} If $a,b\in\Xp$ and $a<b$ then the closed interval $[a,b]$ is covered by countably many elements of $\U$. 
\end{cl}
\begin{proof}
Consider the following subset of $[a,b]$:
\begin{displaymath}
D=\{d\in[a,b]: [a,d] \textrm{ can be covered by countably many }U \in \U\}.
\end{displaymath}
Now, $D\neq\emptyset$, since $a\in D$ ($[a,a]=\{a\}$ is covered by some element of the cover $\U$. )
Since $D$ is a bounded subset of $X$ ($d\leq b$ for any $d\in D$), 
$c=\sup D$ exists. 

We shall show that $c=b$. 
Let us first note that $[a,c]$ can be covered by countably many elements of $\U$. 
The case $c=a$ is obvious.
If $a<c$ we have two cases:

\begin{case}[1]
$c=\sup D$ has an immediate predecessor, i.e. there exists $c_1 \in X$ such that $c_1\lneqq c$ and $(c_1,c)=\emptyset$. 
Then since $c=\sup D$ we have that $c_1\in D$. 
Note that $c_1\nless a$ as then $a\in (c_1,c)=\emptyset$ - a contradiction.
That means that $[a,c_1]$ is covered by countably many elements of $\U$ and hence 
$[a,c]=[a,c_1]\cup \{c\}$ will also be covered by countably many elements of $\U$ 
(at most one more to cover $\{c\}$). 
\end{case}
\begin{case}[2]
$c=\sup D$ has no immediate predecessor. 
Let $\la$ be a cardinal number with $\la\leq|X|$.
Then we can construct an increasing well-ordered sequence $\{c_\al:\al<\la\}$, where $c_\al<c$ for all $\al<\la$ in the following way:
 Let $c_0=a$. 
Since $c$ has no immediate predecessor, there is a $c_1\in (a,c)$, i.e. $c_0<c_1<c$. 
Suppose for some $\al_0 <\la$ we have already constructed a well-ordered $\{c_\al:\al<\al_0\}$
such that $c_\al<c$ for every $\al<\al_0$. 
Since $\{c_\al:\al<\al_0\}$ is bounded by $c$, we can take $c_{\al_0}=\sup\{c_\al:\al<\al_0\}$. 
\end{case}
Then $c_{\al_0}\leq c$. 
If $c_{\al_0}=c$ we are done. 
If $c_{\al_0}<c$ we can carry on the induction up to (possibly) $\la$. 
But this process has to stop at some countable ordinal $\al^{*}$, since by assumption, there is no increasing uncountable subset of $X$. 
So in fact, when $c$ has no immediate predecessor, we can construct $\{c_n:n<\w\}$ which is well-ordered and such that:
$c_n<c$ for all $n\in\w$ and
if $a<e<c$ then there is an $n\in\w$ such that $e\leq c_n$ (i.e. $\{c_n:n\in\w\}$ is cofinal with $c$). 

Let us  note that since $c_n<c$, then $c_n\in D$. 
But $[a,c)=\bigcup_{n<\w}[a,c_n)$, and hence can also be covered by countably many elements of $\U$. 
Then we need possibly one more element of $\U$ to cover $[a,c]=[a,c)\cup\{c\}$. 

Hence $c\in D$, i.e. $[a,c]$ can be covered by countably many elements of $\U$.

If $c=b$, then we are done. 
Suppose that $c<b$.
Again we have two cases to consider:

\begin{case}[1]
$c$ has an immediate successor $c'$, i.e. $c'\in[x,y)$.
Consider $[a,c']=[a,c]\cup\{c\}$.
Then, since $[a,c]$ is countably covered, so is $[a,c']$.
Hence $c'\in D$ and $c'>c=\sup D$ - a contradiction.
\end{case}
\begin{case}[2]
 $c$ has no immediate successor.
Since $c$ is covered by some element $[x,y)\in\U$, $[x,y)$ contains some $c'>c$.
Again - consider $[a,c']=[a,c]\cup[c,c']$.
We know that $[a,c]$ is covered by countably many elements of $\U$ and $[c,c')\subset [x,y]$, hence $[a,c']$ is covered by countably many elements of $\U$.
If $c'\leq b$, then $c'\in D$ and $c'>c=\sup D$ - contradiction.
If $c'>b$ then $[a,b]\subset[a,c']$ hence $[a,b]$ can be countably covered, 
hence $b\in D$ and $c=\sup D < b$ - contradiction.
\end{case}
Hence $c=b$, as required, and hence $[a,b]$ can be covered by countably many elements.
\end{proof}

Now suppose $X$ has no biggest element. 
We shall construct a sequence $\{b_\al:\al<\la\}$ that is cofinal in $X$ 
(that means that for every $x\in X$ there is an $\al<\la$ such that $x< b_\al$ and $b_\al<b_{\al'}$ whenever $\al<\al'<\la$). 
Let $b_0$ be an arbitrary element of $X$ and let $\{b_\al:\al<\al_0\}$ be already constructed such that $b_\al<b_{\al'}$ whenever $\al<\al'$ for some $\al_o < \la$. 
If there is $x\in X$ such that $\be_\al < x$ for every $\al<\al_0$ then $\{\be_\al:\al<\al_0\}$ is a bounded increasing sequence and we may take (by the conditions of the Lemma) $b_{\al_0}=\sup\{b_\al:\al<\al_0\}$ and the induction continues. 
If such $x$ does not exist, then $\{\be_\al:\al<\al_0\}$ is our cofinal sequence (i.e. for all $x\in X$ there exists an $\al_x$ such that $x<\be_{\al_x}$).
By assumptions on $X$, we cannot have an increasing uncountable sequence, so again we must have that this $\la$ is countable, so
$\{b_\al:\al<\la\}$ can be relabeled as $\{b_n:n<\w\}$. 
(Note that if $X$ has a biggest element, then the cofinal sequence $\{b_\al:\al<\la\}$ will have a constant tail.)

Similarly, supposing that $X$ has no smallest element, we can construct a coinitial sequence $\{a_n:n<\w\}$, i.e. a sequence such that for every $x\in X$ there is an $a_n$ such that $a_n<x$ (because in $X$ there are no uncountable ordered decreasing sets). 
Then $\Xp = \bigcup \{[a_n,b_m]:a_n<b_m; n,m<\w\}$. 
Indeed, if $x\in \Xp$ then there are $a_n\leq x$ and $b_m\geq x$ with $x\in [a_n,b_m]$. (Similarly, if $X$ has a smallest element, the coinitial sequence will have a constant tail consisting of just this element.)

But we have proved that any closed interval in $X$ can be countably covered by elements of $\U$. 
Hence, since $\Xp$ is the union of countably many intervals, it can also be covered by countably many elements of $\U$. 

Hence $\Xp$ is Lindel\"of.
\end{proof}

\begin{lem} 
Let $(X,<)$ be a densely ordered set which does not contain a countable dense subset. Then the product space $\Xp\times\Xm$ is not weakly Lindel\"of.
\end{lem}

\begin{proof}

The basic sets in the product topology on $\Xp\times\Xm$ are of the form $[a,b)\times(c,d]$. 

We construct an open cover such that every countable subcover is not dense in $\Xp\times\Xm$. 

Let 
\begin{displaymath}
\D=\{(x,x)\in X\}
\end{displaymath}
be the diagonal of $\Xp\times\Xm$, and let
\begin{displaymath}
\G=\{(x,y)\in \Xp\times\Xm: x<y\}. 
\end{displaymath}
Let us prove that $\G$ is open in $\Xp\times\Xm$. 
So let $(p,q)\in\G$. 
Since $p<q$ and $``<"$ is a dense order on $X$, there exists $r\in X$ with $p<r<q$. 
Then $[p,r)\times(r,q]$ is a neighborhood of $(p,q)$ contained in $\G$. 
 
Now, let 
\begin{displaymath}
[p,\ra)=\{x\in X:x\geq p\}
\end{displaymath}
and
\begin{displaymath}
(\lar,p]=\{x\in X:x\leq p\}.
\end{displaymath}
These sets are open (respectively in $\Xp$ and $\Xm$). 

Consider the following family of open subsets of  $\Xp\times\Xm$:
\begin{displaymath}
\U=\{\G\}\cup\{[p,\ra)\times(\lar,p]:p\in X\}. 
\end{displaymath}
$\U$ is an open cover of  $\Xp\times\Xm$, since if $(x,y)\in \Xp\times\Xm$ we have two cases:
\begin{case}[1]
$x<y$, then $(x,y)\in\G$.
\end{case}
\begin{case}[2]
$x\geq y$, then $(x,y)\in[x\ra)\times(\lar,x]$.
\end{case}
We will prove that for any countable subfamily $\V\subset\U$, the union $V=\bigcup\V$ is not dense in $\Xp\times\Xm$.
In order to make the exposition of Hajnal and Juhasz more transparent, let us note that $\G$ is also closed in $\Xp\times\Xm$, because 
if $(x,y)\notin \G$ then $x\geq y$ and $(x,y)\in [x\ra)\times(\lar,x]$ which is an open subset of $\Xp\times\Xm$ that does not intersect $\G$. 

Let
\begin{displaymath}
A=\{p\in X:[p,\ra)\times(\lar,p]\in\V\}.
\end{displaymath}
Note that if $A=\emptyset$ then $\V=\{\G\}$, and $\bar{\bigcup\V}=\bar{\G}=\G$, hence does not cover $\Xp\times\Xm$.
Now let $A\neq\emptyset$. 
Since $|A|\leq|\V|\leq |\W|$ and $X$ doesn't contain a countable dense subset, $A$ cannot be dense in $(X,<)$.
Hence there is an open interval $(a,b)\in X$ such that $(a,b)\cap A=\emptyset$. 
Since $``<"$ is a dense order, there exists $c\in X$ with $a<c<b$. 
Then $[c,b)\times(a,c]\neq\emptyset$ because $(c,c)\in [c,b)\times(a,c]$.
We have (as in the proof that $\G$ is closed) that $[c,b)\times(a,c]\cap\G=\emptyset$.
Also, for any $p\in A$ we have that 
\begin{displaymath}
([p,\ra)\times(\lar,p])\cap([c,b)\times(a,c])=\emptyset
\end{displaymath}
since if $(r,s)\in ([p,\ra)\times(\lar,p])\cap([c,b)\times(a,c])$ 
then $p\leq r<b$ and $a<s\leq p$ 
so $a<p<b$, a contradiction.

Hence $V\cap [c,b)\times(a,c]=\emptyset$, so we have found a nonempty open subset of $\Xp\times\Xm$ which 
does not intersect $\bigcup\V=V$. 
Hence $\bigcup\V$ is not dense in $\Xp\times\Xm$, hence $\Xp\times\Xm$ is not weakly Lindel\"of.
\end{proof}

Let us also provide a proof of the following result, stated in \cite{SZ}.

\begin{thm} \label{wLtimescpct=wL}If $X$ is weakly Lindel\"of  and $Y$ is compact, then $X \times Y$ is weakly-Lindel\"of. 
\end{thm} 

\begin{proof} 
It is sufficient to prove this statement only in the case when $\U$ is an open cover with basic open sets in $X\times Y$, i.e.
without loss of generality, let $\U\subseteq\{U\times V: U \in \tau_X,V\in\tau_Y\}$ be an open cover for $X\times Y$.
Then for each $x\in X$, the set $\{x\}\times Y$ is compact, and $\U$ is an open cover of $\{x\}\times Y$. Hence $\U$ has a finite subcover $\{U(x,i)\times V(x,i):i=1,\ldots n_x\}$, i.e.
\begin{displaymath}
\{x\}\times Y\subseteq\bigcup_{i\leq n_x} U(x,i)\times V(x,i).
\end{displaymath}
Let $\displaystyle{W_x=\bigcap_{i\leq n_x} U(x,i)}$. 
Then $\W=\{W_x:x\in X\}$ is an open cover of $X$, and $X$ is weakly Lindel\"of. 
Hence there is a countable subfamily $\{W_{x_j}:j\in\N\}$ such that 
\begin{displaymath}
X=\bar{\bigcup_{j\in\N}W_{x_j}}.
\end{displaymath}
Let $\V=\{U(x_j,i)\times V(x_j,i):j\in\N,i=1,\ldots,n_{x_j}\}$.
We shall prove that $X\times Y\subset\bar{\bigcup\V}$.

Let $(s,t)\in X\times Y$ and let $U_s\times V_t$ be a basic open neighborhood of $(s,t)$. 
Since $\displaystyle{s \in X\subseteq\bar{\bigcup_{j\in\N}W_{x_j}}}$,  there is a
$j_0\in\N$ such that $U_s\cap W_{x_{j_0}}\neq\emptyset$.
Then we also have that $U_s\cap U(x_{j_0},i)\neq\emptyset$ for $i=1,\ldots,n_{x_{j_0}}$.
Since $\displaystyle{\{x_{j_0}\}\times Y\subseteq\bigcup_{i\leq n_{x_{j_0}}}(U(x_{j_0},i)\times V(x_{j_0},i))}$, we have that
$\displaystyle{Y=\bigcup_{i\leq n_{x_{j_0}}}V(x_{j_0},i)}$.
Hence there is $i_0\leq n_{x_{j_0}}$ such that $t\in V(x_{j_0},i_0)$.
We then have that $U_s\cap U(x_{j_0},i_0)\neq\emptyset$ and $t\in V_t\cap V(x_{j_0},i_0)$.
Hence
\begin{displaymath}
\left(U_s\times V_t\right)\cap \left(  U(x_{j_0},i_0) \times V(x_{j_0},i_0)\right)\neq\emptyset,
\end{displaymath}
i.e.
\begin{displaymath}
U_s\times V_t\cap\left(\bigcup\V\right)\neq\emptyset,
\end{displaymath}
i.e. $(s,t)\in\bar{\bigcup\V}$. 
Therefore $X\times Y$ is weakly \Lind.
\end{proof}

\subsection{Quasi-Lindel\"of Spaces} \label{quasi-Lindelof}

In \cite{A}, Archangelski introduced the following notion: 

\begin{defn} \label{qL} A topological space is called \emph{quasi-Lindel\"of} if every closed subspace of it is weakly Lindel\"of. 
\end{defn} 

He used this notion to generalize the result of \cite{BGW} that the cardinality of any normal first countable weakly \Lin space is at most continuum.
He proved that any regular first countable quasi-\Lin space has cardinality at most continuum. 
The relation between those two results is based on the fact that every quasi-\Lin space is weakly \Lin and also the following:

\begin{propn} If $X$ is normal and weakly Lindel\"of, then $X$ is quasi-Lindel\"of. 
\end{propn}

\begin{proof}
Let $X$ be a normal weakly Lindel\"of space, and $F$ be a closed subset of $X$. 
Suppose $\G$ is a family of open subsets of $X$, covering $F$, i.e. $\displaystyle{F \subseteq \bigcup\G}$.
Set $U=\bigcup \G$. Then $U$ is open and $U \supseteq F$. 
From normality of $X$, we have that there exists an open set $V \subset X$ such that $F \subseteq V \subseteq \bar{V} \subseteq U$. 
So, consider the following open cover of $X$: $\G_1 = \G \cup \{X \setminus \bar{V}\}$.
This is indeed a cover, as $F \subset \bigcup \G$ and $X \setminus \bar{V} \supseteq X \setminus F$. 
Since $X$ is weakly Lindel\"of, we can choose a countable subfamily $\G' \subseteq \G$ such that $\displaystyle{\left(\bar{\bigcup \G'} \right)\cup \bar{X\setminus \bar{V}}=X}$. 
Now, we have that $F \cap \bar{X\setminus \bar{V}}=\emptyset$:
suppose the opposite, i.e. there is an $x \in F \cap \bar{X\setminus \bar{V}}\subseteq V \cap \bar{X \setminus \bar{V}}$.
Then $x \in V$, which is open, and $x \in \bar{X\setminus \bar{V}}$ - which is closed. 
Hence, $V \cap \bar{X\setminus\bar{V}} \neq \emptyset$. 
But $V \cap (X\setminus \bar{V}) \subseteq \bar{V} \cap (X \setminus \bar{V}) = \emptyset$ - a contradiction.

Therefore, $F$ cannot be covered by $\bar{X \setminus \bar{V}}$.
Hence, $F \subset \bar{\bigcup_{\g \in \G'}\g}$.
\end{proof}

In \cite{BGW} it is stated that every CCC space is weakly \Lind. 
Here, we prove a similar result:

\begin{thm}[PS] If $X$ is CCC then $X$ is quasi-Lindel\"of. 
\end{thm}
\begin{proof}
Suppose that $X$ is CCC but not quasi-Lindel\"of. 
Then there exists a closed nonempty set $F\subset X$ and an uncountable family
$\G=\{U_\al:\al<\be\}$ ($\be\geq\w_1$) of non-empty open in $X$ sets such that
$\displaystyle{F\subset\bigcup_{\al<\be}U_\al}$, but for any countable
$\G'\subset\G$ we have that $F\setminus\bar{\bigcup \G'}\neq\emptyset$.

We will construct an uncountable collection of nonempty disjoint open in $X$ sets $\{V_\g:\g<\w_1\}$, thus contradicting CCC.

Let $V_0=U_0$. 
Then $F\setminus\bar{U_0}\neq\emptyset$ (and $X\setminus \bar{U_0}\neq\emptyset$).
Hence 
\begin{displaymath}
\emptyset\neq F\setminus\bar{U_0}\subset \bigcup\{U_\al:\al<\be,\al>0\}
\end{displaymath}
and therefore
\begin{displaymath}
\emptyset
\neq F\setminus\bar{U_0}
= \bigcup\{U_\al\cap  (F\setminus\bar{U_0}):\al<\be,\al>0\}.
\end{displaymath}
Thus we will have $\al_1\geq1$ such that $U_{\al_1}\cap( F\setminus\bar{U_0})\neq\emptyset$ and moreover 
$U_{\al_1}\cap( X\setminus\bar{U_0})\neq\emptyset$.

Let $V_1=U_{\al_1}\cap( X\setminus\bar{U_0})$.
Then $V_1\neq\emptyset$, $V_1$ is open in $X$ and $V_1\cap V_0=\emptyset$.
Again we have 
\begin{displaymath}
\emptyset\neq F\setminus\bar{U_0\cup U_{\al_1}}=\bigcup\{U_\al\cap(F\setminus \bar{U_0\cup U_{\al_1}});\al<\be, \al\notin\{0,\al_1\}\}.
\end{displaymath}
Hence there is $U_{\al_2}\in\G$, $\al_1\notin\{0,\al_1\}$ with
$U_{\al_2}\cap ( F\setminus\bar{U_0\cup U_{\al_1}})\neq\emptyset$ and moreover
$U_{\al_2}\cap ( X\setminus\bar{U_0\cup U_{\al_1}})\neq\emptyset$.
Define $V_2=U_{\al_2}\cap ( X\setminus\bar{U_0\cup U_{\al_1}})$.
Then $V_2\neq\emptyset$, $V_2$ is open in $X$ and $V_2$ is disjoint from $V_0,V_1$. 
Let $\g_0<\w_1$ and suppose that we have already constructed a family $\{V_\de:\de<\g_0\}$ of non-empty, disjoint open in $X$ sets with 
$V_\de=U_{\al_\de}\cap (X\setminus\bar{\cup\{U_{\al_\si}:\si<\de\}})$, where
$\al_\de\notin\{\al_\si:\si<\de\}$. 
Since $\g_0$ is a countable ordinal and $F$ is not weakly Lindel\"of in $X$, we have that
\begin{displaymath}
\emptyset
\neq F\setminus\bar{\cup\{U_{\al_\de}:\de<\g_0\}}
=\bigcup\{U_\al\cap(F\setminus\bar{\cup\{U_{\al_\de}:\de<\g_0\}}):\al<\de,\al\notin\{\al_\de:\de<\g_0\}\}.
\end{displaymath}
Hence we can choose $\al_{\g_0}$ such that
\begin{displaymath}
U_{\al_{\g_0}}\cap(F\setminus\bar{\cup\{U_{\al_\de}:\de<\g_0\}})\neq\emptyset
\end{displaymath}
and $\al_{\g_0}\notin\{\al_\de:\de<\g_0\}$.
Hence moreover 
\begin{displaymath}
U_{\al_{\g_0}}\cap(X\setminus\bar{\cup\{U_{\al_\de}:\de<\g_0\}})\neq\emptyset.
\end{displaymath}
Define $V_{\g_0}=U_{\al_{\g_0}}\cap(X\setminus\bar{\cup\{U_{\al_\de}:\de<\g_0\}})$. 
Then $V_{\g_0}\neq\emptyset$, $V_{\g_0}$ is open in $X$ and by construction $V_{\g_0}\cap V_\de=\emptyset$ for every $\de<\g_0$.

Thus we have constructed a family $\{V_\g:\g<\w_1\}$ of nonempty disjoint open in $X$ sets, contradicting CCC.
Hence, $X$ is quasi-Lindel\"of. 
\end{proof}

We can also prove that:

\begin{propn}[PS]\label{ctsqL} Continuous images of quasi-Lindel\"of spaces are quasi-Lindel\"of.
\end{propn}

\begin{proof}
Let $f:X\rightarrow Y$ be continuous, $X$ be quasi-Lindel\"of, and $F\subset X$ be closed.  
Since $f^{-1}(F)$ is closed and $X$ is quasi-\Lind, it follows that $f^{-1}(F)$ is weakly \Lind.
By Proposition \ref{ctswL} applied to $f|_{f^{-1}(F)}:f^{-1}(F)\rightarrow F$, we have that $F$ is weakly Lindel\"of.
Hence, $Y$ is quasi-Lindel\"of. 
\end{proof}

\begin{cor} Quotient spaces of quasi-Lindel\"of spaces are quasi-Lindel\"of.
\end{cor}

Let us point out that the product of two quasi-\Lin spaces need not be quasi-\Lind:
\begin{xmpl}
The Hajnal and Juhasz space constructed in the previous section is such an example: it is not quasi-\Lin because it is not weakly \Lind.
\end{xmpl}

\section{Conclusion}

We have outlined basic properties of \Lin spaces and we have compared them with analogous properties for almost \Lind, weakly \Lind, and quasi-\Lin spaces.
We have shown that these three generalizations of \Lind ness in various cases have similar type of behavior as \Lin spaces. 
We have provided examples where such properties fail to be preserved. 
Most, if not all of these examples, cannot be found in textbooks, but rather only in  mathematical journals. 

In view of the present comparison, we conclude with two open questions:

\begin{oq}
Is the product of a quasi-\Lin and a compact space quasi-\Lind?
\end{oq}
(Let us note that such a product will  be weakly \Lin by Theorem \ref{wLtimescpct=wL}.)
\begin{oq}
What kind of topological properties are required in order to show that an almost \Lin space is quasi-\Lin (or respectively, that an almost \Lin space is weakly \Lind)?
\end{oq}

\bibliographystyle{alpha}
\bibliography{Bibliography}
\end{document}

%% file: IrratSlope1pic.tex
%
%
\setlength{\unitlength}{3947sp}%
\begingroup\makeatletter\ifx\SetFigFont\undefined%
\gdef\SetFigFont#1#2#3#4#5{%
  \reset@font\fontsize{#1}{#2pt}%
  \fontfamily{#3}\fontseries{#4}\fontshape{#5}%
  \selectfont}%
\fi\endgroup%
\begin{picture}(7584,4508)(416,-3897)
{\color[rgb]{1,0,0}\thinlines
\put(4542,-1128){\circle{52}}
}%
{\color[rgb]{1,0,0}\put(2852,-2858){\circle{52}}
}%
{\color[rgb]{1,0,0}\put(6242,-2852){\circle{52}}
}%
{\color[rgb]{0,0,0}\put(428,-2861){\vector( 1, 0){7560}}
}%
\put(3058,-3226){\makebox(0,0)[lb]{\smash{{\SetFigFont{10}{12.0}{\rmdefault}{\mddefault}{\updefault}{\color[rgb]{1,0,0}$x-y\2+\e$}%
}}}}
{\color[rgb]{0,.82,0}\put(2860,-2829){\line( 1, 1){1691}}
\put(4538,-1125){\line( 1,-1){1703.500}}
}%
\thicklines
{\color[rgb]{1,0,0}\multiput(2565,-2851)(44.96000,0.00000){13}{\line( 1, 0){ 22.480}}
\put(3127,-2851){\vector( 1, 0){0}}
\put(2565,-2851){\vector(-1, 0){0}}
}%
{\color[rgb]{1,0,0}\multiput(5960,-2851)(44.96000,0.00000){13}{\line( 1, 0){ 22.480}}
\put(6522,-2851){\vector( 1, 0){0}}
\put(5960,-2851){\vector(-1, 0){0}}
}%
\put(4365,-993){\makebox(0,0)[lb]{\smash{{\SetFigFont{11}{13.2}{\rmdefault}{\mddefault}{\updefault}{\color[rgb]{1,0,0}$(x,y)$}%
}}}}
\put(2671,-3030){\makebox(0,0)[lb]{\smash{{\SetFigFont{11}{13.2}{\rmdefault}{\mddefault}{\updefault}{\color[rgb]{1,0,0}$x-y\2$}%
}}}}
\put(6043,-3029){\makebox(0,0)[lb]{\smash{{\SetFigFont{11}{13.2}{\rmdefault}{\mddefault}{\updefault}{\color[rgb]{1,0,0}$x+y\2$}%
}}}}
\put(1957,-3226){\makebox(0,0)[lb]{\smash{{\SetFigFont{10}{12.0}{\rmdefault}{\mddefault}{\updefault}{\color[rgb]{1,0,0}$x-y\2-\e$}%
}}}}
\put(5371,-3256){\makebox(0,0)[lb]{\smash{{\SetFigFont{10}{12.0}{\rmdefault}{\mddefault}{\updefault}{\color[rgb]{1,0,0}$x+y\2-\e$}%
}}}}
\thinlines
{\color[rgb]{0,0,0}\put(1591,-3885){\vector( 0, 1){4484}}
}%
\thicklines
{\color[rgb]{1,0,0}\multiput(2565,-2851)(44.96000,0.00000){13}{\line( 1, 0){ 22.480}}
\put(3127,-2851){\vector( 1, 0){0}}
\put(2565,-2851){\vector(-1, 0){0}}
}%
{\color[rgb]{1,0,0}\thinlines
\put(4542,-1128){\circle{52}}
}%
{\color[rgb]{1,0,0}\put(2852,-2858){\circle{52}}
}%
{\color[rgb]{1,0,0}\put(6242,-2852){\circle{52}}
}%
{\color[rgb]{0,0,0}\put(428,-2861){\vector( 1, 0){7560}}
}%
\put(4365,-993){\makebox(0,0)[lb]{\smash{{\SetFigFont{11}{13.2}{\rmdefault}{\mddefault}{\updefault}{\color[rgb]{1,0,0}$(x,y)$}%
}}}}
{\color[rgb]{0,.82,0}\put(2860,-2829){\line( 1, 1){1691}}
\put(4538,-1125){\line( 1,-1){1703.500}}
}%
\thicklines
{\color[rgb]{1,0,0}\multiput(5960,-2851)(44.96000,0.00000){13}{\line( 1, 0){ 22.480}}
\put(6522,-2851){\vector( 1, 0){0}}
\put(5960,-2851){\vector(-1, 0){0}}
}%
\put(6458,-3257){\makebox(0,0)[lb]{\smash{{\SetFigFont{10}{12.0}{\rmdefault}{\mddefault}{\updefault}{\color[rgb]{1,0,0}$x+y\2+\e$}%
}}}}
\put(2979,-2596){\rotatebox{45.0}{\makebox(0,0)[lb]{\smash{{\SetFigFont{12}{14.4}{\rmdefault}{\mddefault}{\updefault}{\color[rgb]{0,.69,0}$y=\frac{\2}{2}x+c_0$}%
}}}}}
\put(4734,-1231){\rotatebox{315.0}{\makebox(0,0)[lb]{\smash{{\SetFigFont{12}{14.4}{\rmdefault}{\mddefault}{\updefault}{\color[rgb]{0,.69,0}$y=-\frac{\2}{2}x+c_1$}%
}}}}}
\thinlines
{\color[rgb]{0,0,0}\put(1591,-3885){\vector( 0, 1){4484}}
}%
{\color[rgb]{0,.69,0}\put(3058,-2796){\oval( 18, 98)[br]}
\put(3007,-2796){\oval(120,232)[tr]}
}%
\put(3085,-2770){\makebox(0,0)[lb]{\smash{{\SetFigFont{10}{12.0}{\rmdefault}{\mddefault}{\updefault}{\color[rgb]{0,.69,0}$\frac{\pi}{4}$}%
}}}}
\end{picture}%

%% file: Winkler3pic.tex
%
%
\setlength{\unitlength}{3000sp}%
\begingroup\makeatletter\ifx\SetFigFont\undefined%
\gdef\SetFigFont#1#2#3#4#5{%
  \reset@font\fontsize{#1}{#2pt}%
  \fontfamily{#3}\fontseries{#4}\fontshape{#5}%
  \selectfont}%
\fi\endgroup%
\begin{picture}(7816,3894)(222,-3268)
{\color[rgb]{1,0,0}\thinlines
\put(3774,-2949){\circle{76}}
}%
{\color[rgb]{1,0,0}\put(2123,-2950){\circle{76}}
}%
{\color[rgb]{1,0,0}\put(5432,-2942){\circle{76}}
}%
{\color[rgb]{0,0,0}\put(1171,-3174){\vector( 0, 1){3788}}
}%
{\color[rgb]{0,0,0}\put(234,-2941){\vector( 1, 0){7792}}
}%
{\color[rgb]{1,.84,0}\multiput(2124,603)(0.00000,-120.43137){26}{\line( 0,-1){ 60.216}}
\multiput(2124,-2468)(120.00000,0.00000){28}{\line( 1, 0){ 60.000}}
\multiput(5424,-2468)(0.00000,120.43137){26}{\line( 0, 1){ 60.216}}
}%
{\color[rgb]{1,1,1}\put(2137,-2446){\framebox(3264,3058){}}
}%
\put(1910,-3200){\makebox(0,0)[lb]{\smash{{\SetFigFont{12}{14.4}{\rmdefault}{\mddefault}{\updefault}{\color[rgb]{1,0,0}$x-\e$}%
}}}}
\put(5283,-3200){\makebox(0,0)[lb]{\smash{{\SetFigFont{12}{14.4}{\rmdefault}{\mddefault}{\updefault}{\color[rgb]{1,0,0}$x+\e$}%
}}}}
\put(3500,-3200){\makebox(0,0)[lb]{\smash{{\SetFigFont{12}{14.4}{\rmdefault}{\mddefault}{\updefault}{\color[rgb]{1,0,0}$x\in\Q_1$}%
}}}}
{\color[rgb]{1,0,0}\thinlines
\put(3790,-2470){\circle{76}}
}%
{\color[rgb]{1,0,0}\put(4705,-2942){\circle{76}}
}%
{\color[rgb]{1,0,0}\put(4705,-2470){\circle{76}}
}%
{\color[rgb]{0,0,0}\put(1171,-3174){\vector( 0, 1){3788}}
}%
{\color[rgb]{0,0,0}\put(234,-2941){\vector( 1, 0){7792}}
}%
{\color[rgb]{1,.84,0}\multiput(2124,603)(0.00000,-120.43137){26}{\line( 0,-1){ 60.216}}
\multiput(2124,-2468)(120.00000,0.00000){28}{\line( 1, 0){ 60.000}}
\multiput(5424,-2468)(0.00000,120.43137){26}{\line( 0, 1){ 60.216}}
}%
\put(3555,-2689){\makebox(0,0)[lb]{\smash{{\SetFigFont{12}{14.4}{\rmdefault}{\mddefault}{\updefault}{\color[rgb]{1,0,0}$(x,1)$}%
}}}}
\put(4455,-2678){\makebox(0,0)[lb]{\smash{{\SetFigFont{12}{14.4}{\rmdefault}{\mddefault}{\updefault}{\color[rgb]{1,0,0}$(x_1,1)$}%
}}}}
\put(4330,-3200){\makebox(0,0)[lb]{\smash{{\SetFigFont{12}{14.4}{\rmdefault}{\mddefault}{\updefault}{\color[rgb]{1,0,0}$x_1\in\Q_1$}%
}}}}
{\color[rgb]{1,0,0}\thinlines
\put(2845,-2934){\circle{76}}
}%
{\color[rgb]{1,0,0}\put(2845,-340){\circle{76}}
}%
{\color[rgb]{0,0,0}\put(1171,-3174){\vector( 0, 1){3788}}
}%
{\color[rgb]{0,0,0}\put(234,-2941){\vector( 1, 0){7792}}
}%
{\color[rgb]{1,.84,0}\multiput(2124,603)(0.00000,-120.43137){26}{\line( 0,-1){ 60.216}}
\multiput(2124,-2468)(120.00000,0.00000){28}{\line( 1, 0){ 60.000}}
\multiput(5424,-2468)(0.00000,120.43137){26}{\line( 0, 1){ 60.216}}
}%
{\color[rgb]{0,.69,.69}\multiput(1179,-339)(119.18261,0.00000){58}{\line( 1, 0){ 59.591}}
}%
\put(983,-401){\makebox(0,0)[lb]{\smash{{\SetFigFont{12}{14.4}{\rmdefault}{\mddefault}{\updefault}{\color[rgb]{0,0,0}$\al$}%
}}}}
\put(2650,-3200){\makebox(0,0)[lb]{\smash{{\SetFigFont{12}{14.4}{\rmdefault}{\mddefault}{\updefault}{\color[rgb]{1,0,0}$y\in\Q_\al$}%
}}}}
\put(2588,-551){\makebox(0,0)[lb]{\smash{{\SetFigFont{12}{14.4}{\rmdefault}{\mddefault}{\updefault}{\color[rgb]{1,0,0}$(y,\al)$}%
}}}}
\put(961,-2545){\makebox(0,0)[lb]{\smash{{\SetFigFont{12}{14.4}{\rmdefault}{\mddefault}{\updefault}{\color[rgb]{0,0,0}$1$}%
}}}}
\end{picture}%

%% file: Winkler4pic.tex
%
%
\setlength{\unitlength}{3000sp}%
\begingroup\makeatletter\ifx\SetFigFont\undefined%
\gdef\SetFigFont#1#2#3#4#5{%
  \reset@font\fontsize{#1}{#2pt}%
  \fontfamily{#3}\fontseries{#4}\fontshape{#5}%
  \selectfont}%
\fi\endgroup%
\begin{picture}(7816,3821)(222,-3195)
{\color[rgb]{1,.84,0}\thinlines
\put(3190,-2941){\circle{68}}
}%
{\color[rgb]{1,.84,0}\put(3197,-2230){\circle{68}}
}%
{\color[rgb]{1,.84,0}\put(1645,-2942){\circle{68}}
}%
{\color[rgb]{1,.84,0}\put(4727,-2941){\circle{68}}
}%
{\color[rgb]{0,0,0}\put(1171,-3174){\vector( 0, 1){3788}}
}%
{\color[rgb]{0,0,0}\put(234,-2941){\vector( 1, 0){7792}}
}%
{\color[rgb]{1,.84,0}\multiput(1655,603)(0.00000,-120.63830){24}{\line( 0,-1){ 60.319}}
\multiput(1655,-2232)(120.39216,0.00000){26}{\line( 1, 0){ 60.196}}
\multiput(4725,-2232)(0.00000,120.63830){24}{\line( 0, 1){ 60.319}}
}%
\put(2700,-3150){\makebox(0,0)[lb]{\smash{{\SetFigFont{12}{14.4}{\rmdefault}{\mddefault}{\updefault}{\color[rgb]{1,.84,0}$x_1\in\Q_{\al_1}$}%
}}}}
\put(1438,-3150){\makebox(0,0)[lb]{\smash{{\SetFigFont{12}{14.4}{\rmdefault}{\mddefault}{\updefault}{\color[rgb]{1,.84,0}$x_1-\e_1$}%
}}}}
\put(4492,-3150){\makebox(0,0)[lb]{\smash{{\SetFigFont{12}{14.4}{\rmdefault}{\mddefault}{\updefault}{\color[rgb]{1,.84,0}$x_1+\e_1$}%
}}}}
\put(2896,-2437){\makebox(0,0)[lb]{\smash{{\SetFigFont{12}{14.4}{\rmdefault}{\mddefault}{\updefault}{\color[rgb]{1,.84,0}$(x_1,\al_1)$}%
}}}}
\put(900,-2293){\makebox(0,0)[lb]{\smash{{\SetFigFont{12}{14.4}{\rmdefault}{\mddefault}{\updefault}{\color[rgb]{0,0,0}$\al_1$}%
}}}}
{\color[rgb]{0,0,1}\thinlines
\put(5680,-2934){\circle{68}}
}%
{\color[rgb]{0,0,1}\put(3782,-2935){\circle{68}}
}%
{\color[rgb]{0,0,1}\put(7562,-2935){\circle{68}}
}%
{\color[rgb]{0,0,1}\put(5672,-2702){\circle{68}}
}%
{\color[rgb]{0,0,0}\put(1171,-3174){\vector( 0, 1){3788}}
}%
{\color[rgb]{0,0,0}\put(234,-2941){\vector( 1, 0){7792}}
}%
{\color[rgb]{0,0,1}\multiput(3799,603)(0.00000,-120.03636){28}{\line( 0,-1){ 60.018}}
\multiput(3799,-2698)(119.39683,0.00000){32}{\line( 1, 0){ 59.698}}
\multiput(7560,-2698)(0.00000,120.03636){28}{\line( 0, 1){ 60.018}}
}%
\put(3700,-3150){\makebox(0,0)[lb]{\smash{{\SetFigFont{12}{14.4}{\rmdefault}{\mddefault}{\updefault}{\color[rgb]{0,0,1}$x_2-\e_2$}%
}}}}
\put(7359,-3126){\makebox(0,0)[lb]{\smash{{\SetFigFont{12}{14.4}{\rmdefault}{\mddefault}{\updefault}{\color[rgb]{0,0,1}$x_2+\e_2$}%
}}}}
\put(5491,-3131){\makebox(0,0)[lb]{\smash{{\SetFigFont{12}{14.4}{\rmdefault}{\mddefault}{\updefault}{\color[rgb]{0,0,1}$x_2\in\Q_{\al_2}$}%
}}}}
\put(5376,-2615){\makebox(0,0)[lb]{\smash{{\SetFigFont{12}{14.4}{\rmdefault}{\mddefault}{\updefault}{\color[rgb]{0,0,1}$(x_2,\al_2)$}%
}}}}
\put(900,-2757){\makebox(0,0)[lb]{\smash{{\SetFigFont{12}{14.4}{\rmdefault}{\mddefault}{\updefault}{\color[rgb]{0,0,0}$\al_2$}%
}}}}
{\color[rgb]{0,.69,0}\thinlines
\put(4017,-2941){\circle{68}}
}%
{\color[rgb]{0,.69,0}\put(4253,-1531){\circle{68}}
}%
{\color[rgb]{0,.69,0}\put(4489,-2941){\circle{68}}
}%
{\color[rgb]{0,.69,0}\put(4253,-2941){\circle{68}}
}%
{\color[rgb]{0,0,0}\put(1171,-3174){\vector( 0, 1){3788}}
}%
{\color[rgb]{0,0,0}\put(234,-2941){\vector( 1, 0){7792}}
}%
{\color[rgb]{0,.69,0}\multiput(4017,603)(0.00000,-121.48571){18}{\line( 0,-1){ 60.743}}
\multiput(4017,-1523)(134.85714,0.00000){4}{\line( 1, 0){ 67.429}}
\multiput(4489,-1523)(0.00000,121.48571){18}{\line( 0, 1){ 60.743}}
}%
\put(900,-1574){\makebox(0,0)[lb]{\smash{{\SetFigFont{12}{14.4}{\rmdefault}{\mddefault}{\updefault}{\color[rgb]{0,0,0}$\al$}%
}}}}
\put(4010,-3600){\makebox(0,0)[lb]{\smash{{\SetFigFont{12}{14.4}{\rmdefault}{\mddefault}{\updefault}{\color[rgb]{0,.69,0}$x\in\Q_\al$}%
}}}}
\put(3650,-3365){\makebox(0,0)[lb]{\smash{{\SetFigFont{12}{14.4}{\rmdefault}{\mddefault}{\updefault}{\color[rgb]{0,.69,0}$x-\e$}%
}}}}
\put(4350,-3365){\makebox(0,0)[lb]{\smash{{\SetFigFont{12}{14.4}{\rmdefault}{\mddefault}{\updefault}{\color[rgb]{0,.69,0}$x+\e$}%
}}}}
\end{picture}%

%% file: Winkler5pic.tex
%
%
\setlength{\unitlength}{3500sp}%
\begingroup\makeatletter\ifx\SetFigFont\undefined%
\gdef\SetFigFont#1#2#3#4#5{%
  \reset@font\fontsize{#1}{#2pt}%
  \fontfamily{#3}\fontseries{#4}\fontshape{#5}%
  \selectfont}%
\fi\endgroup%
\begin{picture}(7816,3852)(222,-3226)
{\color[rgb]{0,1,0}\thinlines
\put(1645,-2942){\circle{68}}
}%
{\color[rgb]{0,1,0}\put(4727,-2941){\circle{68}}
}%
{\color[rgb]{0,0,0}\put(1171,-3174){\vector( 0, 1){3788}}
}%
{\color[rgb]{0,0,0}\put(234,-2941){\vector( 1, 0){7792}}
}%
{\color[rgb]{0,.69,.69}\multiput(1179,-933)(120.92035,0.00000){57}{\line( 1, 0){ 60.460}}
}%
\put(900,-2293){\makebox(0,0)[lb]{\smash{{\SetFigFont{12}{14.4}{\rmdefault}{\mddefault}{\updefault}{\color[rgb]{0,0,0}$\al_1$}%
}}}}
\put(1522,-3149){\makebox(0,0)[lb]{\smash{{\SetFigFont{12}{14.4}{\rmdefault}{\mddefault}{\updefault}{\color[rgb]{0,1,0}$a$}%
}}}}
\put(4579,-3130){\makebox(0,0)[lb]{\smash{{\SetFigFont{12}{14.4}{\rmdefault}{\mddefault}{\updefault}{\color[rgb]{0,1,0}$b$}%
}}}}
\put(900,-2026){\makebox(0,0)[lb]{\smash{{\SetFigFont{12}{14.4}{\rmdefault}{\mddefault}{\updefault}{\color[rgb]{0,0,0}$\al_2$}%
}}}}
\put(900,-1554){\makebox(0,0)[lb]{\smash{{\SetFigFont{12}{14.4}{\rmdefault}{\mddefault}{\updefault}{\color[rgb]{0,0,0}$\al_n$}%
}}}}
\put(30,-951){\makebox(0,0)[lb]{\smash{{\SetFigFont{12}{14.4}{\rmdefault}{\mddefault}{\updefault}{\color[rgb]{0,.69,.69}$\sup \al_n = \al_0$}%
}}}}
\put(570,-649){\makebox(0,0)[lb]{\smash{{\SetFigFont{12}{14.4}{\rmdefault}{\mddefault}{\updefault}{\color[rgb]{0,0,0}$\al_0+1$}%
}}}}
\thinlines
{\color[rgb]{0,0,0}\put(234,-2941){\vector( 1, 0){7792}}
}%
{\color[rgb]{0,1,0}\multiput(1645,603)(0.00000,-122.84211){10}{\line( 0,-1){ 61.421}}
\multiput(1645,-564)(120.86275,0.00000){26}{\line( 1, 0){ 60.431}}
\multiput(4727,-564)(0.00000,122.84211){10}{\line( 0, 1){ 61.421}}
}%
\put(2620, 21){\makebox(0,0)[lb]{\smash{{\SetFigFont{14}{16.8}{\rmdefault}{\mddefault}{\updefault}{\color[rgb]{0,1,0}$O(a,b,\al)$}%
}}}}
{\color[rgb]{0,0,0}\put(1171,-3174){\vector( 0, 1){3788}}
}%
\end{picture}%